\documentclass[11pt]{amsart}
\numberwithin{equation}{section}
\usepackage[english]{babel}
\usepackage[T1]{fontenc}
\usepackage{indentfirst}
\usepackage{enumitem}
\usepackage{amsmath,amssymb, amsbsy}
\usepackage{comment}
\usepackage{amsfonts}
\usepackage{hyperref}
\usepackage{cleveref}
\usepackage{esint}
\usepackage{latexsym}
\usepackage{amsthm}
\usepackage{bm}
\usepackage[dvips]{graphicx}
\usepackage{xcolor}
\usepackage{tikz}
\usepackage{pgfplots}
\usetikzlibrary{intersections}
\usepackage{tikz-3dplot}
\usepackage[outline]{contour}
\DeclareGraphicsExtensions{.pdf,.png,.jpg,.eps}
\usepackage{amsmath}
\usepackage{mathdots}
\usepackage{yhmath}
\usepackage{cancel}
\usepackage{color}
\usepackage{siunitx}
\usepackage{array}
\usepackage{multirow}
\usepackage{amssymb}
\usepackage{textcomp}
\usepackage{gensymb}
\usepackage{tabularx}
\usepackage{extarrows}
\usepackage{booktabs}
\usetikzlibrary{fadings}
\usetikzlibrary{patterns}
\usetikzlibrary{shadows.blur}
\usetikzlibrary{shapes}
\usepackage{bbm}
\usepackage[font=small,labelfont=bf]{caption}
\usepackage[utf8]{inputenc}
\usepackage[active]{srcltx}

\allowdisplaybreaks

\usepackage{doi}
\newcommand{\arxiv}[1]{arXiv:\textcolor{magenta}{\href{https://arxiv.org/abs/#1}{#1}}}
\let\arXiv\arxiv

\newcommand{\mf}[1]{\mathbf{#1}}

\newcommand{\be}{\begin{equation}}
	\newcommand{\ee}{\end{equation}}
\newcommand{\mathbbmm}[1]{\text{\usefont{U}{bbm}{m}{n}#1}}
\newcommand{\ind}{\mathbbmm{1}}

\newcommand{\brd}[1]{\mathbb{#1}}
\newcommand{\R}{\brd{R}}

\newcommand{\cG}{\mathcal{G}}
\newcommand{\cO}{\mathcal{O}}
\newcommand{\cV}{\mathcal{V}}

\newcommand{\N}{\brd{N}}
\newcommand{\Z}{\brd{Z}}
\newcommand{\e}{\varepsilon}
\newcommand{\pa}{\partial}
\newcommand{\supp}{\operatorname{spt}}

\newcommand{\D}{\nabla}

\newcommand{\dist}{\operatorname{dist}}

\newcommand{\loc}{{\rm loc}}
\renewcommand{\u}{\mathbf{u}}

\def\O{\Omega}

\newcommand{\norm}[2]{\left\Vert {#1} \right\Vert_{#2}}
\newcommand{\eps}{\varepsilon}
\newcommand\ddfrac[2]{\frac{\displaystyle #1}{\displaystyle #2}}

\newtheorem{teo}{Theorem}[section]
\newtheorem{Corollary}[teo]{Corollary}
\newtheorem{Lemma}[teo]{Lemma}
\newtheorem{Theorem}[teo]{Theorem}
\newtheorem{Proposition}[teo]{Proposition}
\newtheorem{proposition}[teo]{Proposition}
\newtheorem{lemma}[teo]{Lemma}
\theoremstyle{definition}
\newtheorem{Definition}[teo]{Definition}

\newtheorem{remark}[teo]{Remark}

\pgfplotsset{compat=1.18}


\title[Optimal partition and segregation problems driven by torsional rigidity]{Optimal partition and segregation problems\\ driven by torsional rigidity}

\author{Gabriele Fioravanti, Nicola Soave, and Giorgio Tortone}

\address{Gabriele Fioravanti, Nicola Soave, and Giorgio Tortone\newline\indent Dipartimento di Matematica ``G. Peano'' \newline\indent Universit\`a degli Studi di Torino \newline\indent Via Carlo Alberto 10, 10124, Torino, Italy} \email{gabriele.fioravanti@unito.it, nicola.soave@unito.it, giorgio.tortone@unito.it}

\begin{document}

\begin{abstract}
Spectral optimal partition and segregation problems are deeply connected with harmonic maps, eigenfunctions, and the fine structure of nodal sets for linear elliptic equations. In this paper, we show that replacing the spectral energy by torsional rigidity leads to a genuinely different theory. The resulting optimal configurations are governed locally not by harmonic equations, but by torsion-type energies and unstable free boundary problems, thereby creating a natural bridge between optimal partition theory and the analysis of sublinear free boundary phenomena.

We prove existence of optimal torsional partitions and segregated torsional configurations, together with optimal Lipschitz regularity of the associated nonlinear eigenfunctions. We establish a strong unique continuation principle, characterize the admissible vanishing orders and the corresponding blow-up profiles, derive sharp Hausdorff dimension estimates for the nodal set and its singular subset, and prove $C^{1,\alpha}$-regularity of the regular part of the free boundary. The proofs combine variational arguments with Almgren-type and Weiss-type monotonicity formulae adapted to the intrinsically sublinear torsional regime, blow-up analysis, and tools from geometric measure theory.
\end{abstract}

\maketitle

\section{Introduction}

Optimal partition problems for the eigenvalues of the Laplacian, harmonic maps into singular spaces, and variational systems modelling spatial segregation have attracted considerable research interest in recent decades. Paradigmatic examples are provided by the following problems.

\smallskip

 \textbf{Problem 1: Optimal spectral partitions.} Given a bounded smooth domain $\Omega \subset \mathbb{R}^N$, study
\be\label{spec part prob}
 \inf\left\{
 \sum_{i=1}^k \lambda_1(\omega_i) \ \left|\
 \begin{array}{l}
 (\omega_1,\dots,\omega_k) \ \text{with }\omega_i \subset \Omega \text{ open and non-empty},\\
 \omega_i \cap \omega_j = \emptyset \quad \text{for all } i \neq j
 \end{array}
 \right.
 \right\}.
\ee
Here, $\lambda_1(\omega_i)$ denotes the first Dirichlet eigenvalue of the Laplacian on $\omega_i$.

\smallskip

\textbf{Problem 2: Multi-valued harmonic functions with free boundaries.} Given a bounded smooth domain $\Omega \subset \mathbb{R}^N$ and
$k$ non-negative and non-trivial functions $\varphi_1,\dots,\varphi_k \in C^{0,1}(\partial\Omega)$ satisfying the segregation condition
$\varphi_i \cdot \varphi_j \equiv 0$ on $\partial\Omega$ for $i\neq j$, study
\be\label{harm part prob}
 \inf\left\{
 \sum_{i=1}^k \int_{\Omega} |\nabla u_i|^2 \, dx \ \left|\
 \begin{array}{l}
 u_i-\varphi_i \in H^1_0(\Omega),\\
 u_i \cdot u_j \equiv 0 \quad \text{in } \Omega \text{ for all } i\neq j
 \end{array}
 \right.
 \right\}.
\ee

Thanks to the contributions of several research groups, the underlying theory is now well developed. We refer to \cite{Alp, BuBu, BuBuHe, CafLin2, ConTerVer05', OgVe1, OgVe, OgVe2, RaTaTe, TaTe12} for Problem 1 and to \cite{Alp, CafLin1,ConTerVer05,GrSc, OgVe, SoTe15, TaTe12} for Problem 2, as well as to the references therein. In particular, it is known that both problems can be addressed by means of a common strategy, originally developed in the study of nodal sets of harmonic functions and, more generally, of solutions to linear elliptic equations.

The intention of this paper is to address variational partition and segregation problems in a different regime, where the local behavior is dictated by torsion-type energies rather than harmonic ones.
More precisely, we shall study the following problems.

\smallskip

 \textbf{Problem A: Optimal torsional partitions.} Given a bounded smooth domain $\Omega \subset \mathbb{R}^N$, study
\be\label{eq:partition.min}
 \mathcal{L}:=\inf\left\{
 \sum_{i=1}^k \lambda_{2,1}(\omega_i) \ \left|\
 \begin{array}{l}
 \omega_i \subset \Omega \quad \text{open and non-empty},\\
 \omega_i \cap \omega_j = \emptyset \quad \text{for all } i \neq j
 \end{array}
 \right.
 \right\},
\ee
where
\[
 \lambda_{2,1}(\omega)
 :=
 \inf\left\{
 \int_{\omega} |\nabla u|^2\,dx
 \ \left|\
 u \in H^1_0(\omega), \ \|u\|_{L^1(\omega)} = 1
\right. \right\}.
\]
Here, $\lambda_{2,1}(\omega)$ is the optimal constant in the Poincar\'e inequality for the embedding $H^1_0(\omega)\hookrightarrow L^1(\omega)$. Equivalently, $\lambda_{2,1}(\omega)=1/T(\omega)$
where $T(\omega)$ denotes the torsional rigidity of $\omega$.
We refer to \cite{BrDPFr, BrFr, BrFr2, BrFrRu} and to references therein for characterization and properties of $\lambda_{2,1}$.

\smallskip

 \textbf{Problem B: Multi-valued torsional functions with free boundaries.} Given a bounded smooth domain $\Omega \subset \mathbb{R}^N$, $k$ positive numbers $\lambda_1,\dots,\lambda_k \in (0,+\infty)$, and $k$ non-negative and non-trivial functions $\varphi_1,\dots,\varphi_k \in C^{0,1}(\partial\Omega)$ satisfying the segregation condition
$\varphi_i \cdot \varphi_j \equiv 0$ on $\partial\Omega$ for $i\neq j$, study
\be\label{pbB}
 \inf\left\{
 \sum_{i=1}^k \int_{\Omega} \left( |\nabla u_i|^2 - 2\lambda_i u_i \right)\,dx
 \ \left|\
 \begin{array}{l}
 u_i-\varphi_i \in H^1_0(\Omega),\\
 u_i \cdot u_j \equiv 0 \quad \text{in } \Omega \text{ for all } i\neq j
 \end{array}
 \right.
 \right\}.
\ee
Note that, if $(u_1,\dots,u_k)$ is a minimizer, then $-\Delta u_i=\lambda_i$ in its positivity set. Thus, the problem is a natural multi-phase counterpart of the classical torsion problem, with prescribed segregated boundary traces. 

\medskip

Although the difference between Problem A and Problem 1, and between Problem B and Problem 2, may at first appear subtle, it has far-reaching consequences. In the torsional setting, the optimal configurations are governed by a different local geometry, which separates them sharply from the harmonic configurations arising in Problems 1 and 2. As a consequence, most of the techniques developed for Problems 1 and 2 do not extend to Problems A and B, and new ideas and tools are required.

The reason is structural. Problems 1 and 2 are governed by techniques related to nodal sets of harmonic functions and, more generally, to solutions of linear elliptic equations. Problems A and B, on the other hand, are tied to unstable free boundary problems and to nodal sets generated by sublinear equations. This places the analysis in a different regime and leads to new geometric and analytic phenomena.

\subsection{Main results for Problem A} In what follows, we present our main results for Problem A, highlighting the main differences with respect to Problem 1 and briefly describing the strategies used to overcome the main obstacles. We will always implicitly assume that $N \ge 2$ (the case $N=1$ can be treated by explicit calculations).

The first result concerns the existence of a solution. A possible approach consists in adapting the general strategy developed in \cite{BuBuHe} (see also \cite{BuBu}), which provides the existence of an optimal partition in the class of \emph{quasi-open} sets. On the other hand, in order to obtain an open partition, it is convenient to argue as in \cite{CafLin2,ConTerVer05'}, and reformulate the original shape optimization problem as a minimization problem for segregated vector-valued functions: given $k\in \N$, we set
\[
\Sigma_k := \left\{ \xi \in \R^k \ \left| \ \sum_{i\neq j} \xi_i^2 \xi_j^2=0 \right.\right\},
\]
and we denote by $H^1_0(\O,\Sigma_k)$ the class of vector-valued function $\u\in H^1_0(\O,\R^k)$ such that $\u(x) \in \Sigma_k$ for almost every $x \in \Omega$, namely $u_i u_j \equiv 0$ a.e. in $\O$ for every $i\neq j$. Now we consider the minimization problem
\be\label{e:partition.functions}
\ell:= 
\inf\left\{ \int_\O |\nabla \u|^2\,dx \ \left|  \ \u \in H^1_0(\O,\Sigma_k) \text{ such that }\mf{u} \ge 0 \text{ and }\norm{u_i}{L^1(\O)} = 1\text{ for every $i=1,\dots,k$}\right.\right\}.
\ee

\begin{Theorem}\label{t:existence-equivalence} The following statements hold true.
\begin{enumerate}
    \item [($i$)] The minimization problem \eqref{e:partition.functions} has a solution $\u\in H^1_0(\O,\Sigma_k)$.
    \item [($ii$)]  Any minimizer of \eqref{e:partition.functions} is Lipschitz continuous in $\overline{\Omega}$.
    \item [($iii$)] Problems \eqref{eq:partition.min} and \eqref{e:partition.functions} are equivalent, in the sense that $\mathcal{L}=\ell$; if $\u$ is a solution to \eqref{e:partition.functions}, then $(\{u_1>0\},\dots,\{u_k>0\})$ is a solution to \eqref{eq:partition.min}; and, conversely, if $(\omega_1,\dots,\omega_k)$ solves \eqref{eq:partition.min}, then $\u=(u_1,\dots,u_k)$ defined by
    \[
    u_i := \mathrm{argmin}\left\{\int_{\omega_i} |\nabla u|^2\,dx \ \left| \  \text{$u\in H_0^1(\omega_i)$ such that $u \ge 0$ and $\|u\|_{L^1(\omega_i)}=1$}\right.\right\},
    \]
    solves \eqref{e:partition.functions}. In particular, \eqref{eq:partition.min} has a solution. 
\end{enumerate}
\end{Theorem}

The main advantages of reformulation \eqref{e:partition.functions} are twofold. First, it allows us to recast the existence of an open solution to \eqref{eq:partition.min} in terms of the regularity of segregated functions, a question that can be addressed through the approach developed in \cite{ConTerVer05}. Furthermore, it provides the extremality conditions for the minimal density collected in Section \ref{sec: basics} below, which turn out to be decisive in the subsequent regularity analysis of the free boundary. In particular, we mention the following:

\begin{Proposition}\label{p:classeS}
Let $\u \in H^1_0(\O,\Sigma_k)$ be a minimizer of \eqref{e:partition.functions} and $i=1,\dots,k$. Then 
\[
-\Delta u_i \leq \lambda_{2,1}(\omega_i)\ind_{\{u_i>0\}},\quad 
  -\Delta \bigg( u_{i}-\sum_{j\neq i}u_{j}\bigg) \ge \lambda_{2,1}(\omega_i)\ind_{\{u_{i}>0\}\cup\{\u=0\}} - \sum_{j\neq i}\lambda_{2,1}(\omega_j)\ind_{\{u_{j}>0\}}\quad \text{in $\O$},
\]
where the previous inequalities are understood in the weak sense. 
\end{Proposition}

The proposition suggests that the reference scalar equation for our vector-valued problem is not a linear equation, as in the case of spectral partitions, but rather 
\be\label{2pop}
-\Delta w=\lambda_+ \ind_{\{w> 0\}} - \lambda_- \ind_{\{w<0\}},
\ee
which is the equation of the \emph{unstable two-phase obstacle problem}, studied in \cite{SoaTer18}. This can be interpreted as the limiting case of the family of sublinear equations 
\be\label{sube}
-\Delta w=\lambda_+ (w^+)^{q-1} - \lambda_- (w^-)^{q-1}, \quad q \in (1,2),
\ee
as $q \to 1^+$, for which we refer to \cite{SoaTer18} (see also \cite{AnShWe2, MoWe, Ru, SoaTor24, SoaWet18} for strictly related results). Unlike in the linear case, an Almgren monotonicity formula is not available for \eqref{2pop} and \eqref{sube}; such a formula proves to be a fundamental tool in the analysis of the nodal set and of the free-boundary, defined by 
\[
Z(\u):=\{\u=0\} \cap \Omega \quad \text{and} \quad \Gamma(\u):=\bigcup_{i=1}^k \pa\{u_i>0\} \cap \Omega,
\]
respectively. Consequently, this study requires a substantially different approach compared to the linear case, as already extensively observed in the aforementioned contributions. In our case, the first result in the nodal analysis is a strong unique continuation principle.

\begin{Definition}\label{D:VOi}
Let $\u \in H^1(\Omega,\Sigma_k)$ and $x_0 \in \Omega$. We define the \emph{vanishing order of $\u$ at $x_0$} as
\[
\mathcal{O}(\mf{u},x_0):= \sup\left\{\gamma>0\left| \ \limsup_{r \to 0^+} \frac{\|\u\|_{x_0,r}}{r^{\gamma}} <+\infty \right. \right\},
\]
where
\[
\|u\|_{x_0,r}^2:= \frac{1}{r^{N-2}} \int_{B_r(x_0)} |\nabla u|^2\,dx + \frac{1}{r^{N-1}} \int_{\partial B_r(x_0)} u^2\,d\sigma \quad \text{and} \quad \|\mf{u}\|_{x_0,r}^2:= \sum_{i=1}^k \|u_i \|_{x_0,r}^2.
\]
\end{Definition}

\begin{remark}
    By the Poincar\'e inequality
    \be\label{Poinc norm}
\frac{1}{r^N} \int_{B_r(x_0)} u^2\,dx \le \frac{1}{N-1} \left( \frac{1}{r^{N-2}} \int_{B_r(x_0)} |\nabla u|^2\,dx + \frac{1}{r^{N-1}} \int_{\partial B_r(x_0)} u^2\,d\sigma\right),
\ee
the norm $\|\cdot\|_{x_0,r}$ is equivalent to the usual $H^1$-norm on $B_r(x_0)$, in its scale-invariant form.
\end{remark}

\begin{Theorem}\label{prop: non-deg}
If $\u\in H^1_0(\O,\Sigma_k)$ is a minimizer of \eqref{e:partition.functions}, then $\u$ is \emph{$2$-non-degenerate}, in the sense that 
\[
\liminf_{r \to 0^+} \frac{\|\u\|_{x_0,r}}{r^2} >0 \qquad \text{for all } x_0 \in Z(\u).
\]
In particular, $\cO(\mf{u},x_0) \in [1, 2]$ for every $x_0 \in Z(\u)$, the set $Z(\u)$ has empty interior, and $Z(\u) = \Gamma(\u)$.
\end{Theorem}

This result provides universal bounds on the admissible vanishing orders for nonlinear eigenfunctions, and establishes that any optimal partition exhausts the whole domain $\Omega$, namely $\bigcup_i \overline{\omega_i}=\overline{\Omega}$. This is a non-trivial fact, since in Problem A such a condition is not required (see \eqref{eq:partition.min}). As observed in \cite{SoaTer18}, the existence of universal bounds on admissible vanishing orders is a distinctive feature of sublinear equations, in sharp contrast with the linear case and, consequently, with spectral partition problems of type \eqref{spec part prob}.

The next substantial step is the exact characterization of the admissible vanishing orders, together with the blow-up analysis at points of the nodal set. Once again, this step requires ingredients that differ from those used in the spectral case, and the sublinear nature of the problem becomes even more apparent. 

This is reflected in the following dichotomy. At points with vanishing order $\cO(\u,x_0) \in [1,2)$, we prove perturbed Almgren-type monotonicity formulae. The corresponding limiting profiles are $\cO(\u,x_0)$-homogeneous and harmonic-like; namely, they belong to the same class of admissible blow-up profiles as in problem \eqref{spec part prob}. By contrast, at points with quadratic vanishing order - the natural scaling order for equation  \eqref{2pop} - an Almgren formula does not seem to be available. Instead, a two-parameter family of Weiss-type monotonicity formulae turns out to be useful. This approach is inspired by \cite{SoaTer18}, but it has to be suitably adapted to the multi-phase setting. This allows one to prove the convergence of blow-up sequences to a different class of homogeneous limiting profiles, related to equation \eqref{2pop}. The precise statements related to the blow-up analysis are postponed to Section \ref{sec: blow-up}, since they are rather lengthy and require additional notation.

Here, we highlight another fundamental ingredient of the analysis, which is also of independent interest. In order to classify the admissible profiles and derive useful information for the study of the free boundary, it is essential to transfer the minimality property from $\u$, a minimizer of \eqref{e:partition.functions}, to its blow-up limits. 
This passage, however, is particularly delicate and, in fact, impossible in this form, due to the presence of the $L^1$-normalizations in the definition of problem \eqref{e:partition.functions}. To overcome this obstacle, we prove that solutions of \eqref{e:partition.functions} can be locally characterized as \emph{almost-minimizers} of an unconstrained problem, as follows.

\begin{proposition}\label{p:almost}
    Let $\u\in H^1_0(\O,\Sigma_k)$ be a minimizer of \eqref{e:partition.functions}, $\Lambda:= (\lambda_{2,1}(\{u_1>0\}),\dots,\lambda_{2,1}(\{u_k>0\}))$, and $K \subset \Omega$ be compact. There exist $C>0$ and $r_0>0$ such that, for every $x_0 \in Z(\u) \cap K$ and every $r\in(0,r_0)$, the following almost-minimality condition is satisfied:
\be\label{e:almost-minimality}
\int_{B_r(x_0)} \Big(|\nabla \u|^2 - 2 (\Lambda \cdot \u)\Big)\,dx \leq  \int_{B_r(x_0)}  \Big(|\nabla \mathbf{w}|^2 - 2 (\Lambda \cdot \mathbf{w})\Big)\,dx  + C r^{2N+1},
\ee
for every $\mathbf{w} \in H^1(B_r(x_0),\Sigma_k)$ such that $\mf w=\u$ on $\partial B_r(x_0)$.
\end{proposition}

With the results obtained in the blow-up analysis at our disposal, we can proceed with the study of the regularity of the free boundary. First, in view of the harmonic-like local behavior of minimizers close to any nodal point, we inherit from \cite{CafLin1, TaTe12, SoTe15} a universal gap between the lower admissible vanishing order and the higher ones. This allows us to conveniently define a closed \emph{singular subset} of $Z(\u)$. 

Second, by using Federer's dimension reduction principle, we obtain optimal estimates on the Hausdorff dimension of the nodal set and of its singular subset. Finally, we study the regularity of the regular set by means of geometric measure theory techniques, suitably combined with the Euler-Lagrange equations in Proposition \ref{p:classeS}. Collecting the previous results, we obtain the following statement.

\begin{Theorem}\label{thm: fb}
    Let $\u$ be a minimizer of \eqref{e:partition.functions}. Then the following holds true.
    \begin{enumerate}
        \item [($a$)] Let $x_0 \in Z(\u)$. Then, either $\cO(\u,x_0) =1$ or $\cO(\u,x_0) \ge 3/2$. The singular set
    \[
    \Sigma(\u):= \{x_0 \in Z(\u) \ \left| \ \cO(\u,x_0)>1\right.\},
    \]
    is a relatively closed subset of $Z(\u)$, and:
    \begin{enumerate}
        \item[($i$)] $\mathrm{dim}_{\mathcal{H}}(Z(\u)) \le N-1$;
        \item[($ii$)] $\mathrm{dim}_{\mathcal{H}}(\Sigma(\u))  \le N-2$, and $\Sigma(\u)$ is a discrete set if $N=2$.
    \end{enumerate} 
        \item [($b$)] The regular set $\mathcal R(\u) := Z(\u) \setminus \Sigma(\u)$ is locally a $C^{1,\alpha}$-regular hypersurface, for every $\alpha \in (0,1)$; more precisely, if $x_0 \in \mathcal{R}(\u)$, then there exist two indices $i \neq j$ in $\{1,\dots,k\}$ and $r>0$ such that
        \[
        u_i\not\equiv0, \ u_j\not\equiv0  \ \text{ in }B_r(x_0), \quad \text{and } \quad u_h\equiv0 \ \text{ in }B_r(x_0) \quad \text{for every }h\not=i,j, 
        \]
    $w:= u_i-u_j$ solves \eqref{2pop} in $B_r(x_0)$ with $\lambda_+=\lambda_{2,1}(\{u_i>0\})$, $\lambda_-=\lambda_{2,1}(\{u_j>0\})$, and $\nabla w(x_0) \neq 0$.
    \end{enumerate}
\end{Theorem}

Thus, despite the fact that the analysis in the case under consideration is substantially more involved than in the spectral case \eqref{spec part prob}, one can recover the same optimal estimates on the dimensions of the nodal and singular sets, and on the regularity of the free boundary.

We conclude the presentation of the results concerning Problem A with two additional questions.

\begin{remark}\label{rem: further}
$1$) In light of the previous theorem, a first question concerns the properties of the possible limiting configurations associated with singular points. Focusing on the two-dimensional case ($N=2$), for the classical spectral problem \eqref{spec part prob} it is known that all admissible limiting profiles have half-integer vanishing orders and, up to rotations, the densities are given by the restriction of $\operatorname{Re}(z^d)$ to certain connected components of its positivity set, with $d \in \mathbb Z/2$. We refer to \cite{ConTerVer05', TaTe12} and also to \cite[Theorem 1.4]{SoTe15} for more details. This implies that, in dimension $N=2$, the free boundary consists of smooth arcs meeting at singular points \emph{with equal angles}.

For Problem A, the picture is different in several aspects. In view of the upper bound on the possible vanishing orders, it is natural to ask whether this universal bound entails a corresponding rigidity of the admissible profiles - for instance, whether they might be finite in number. This is indeed the case for harmonic-like admissible blow-up limits (see Section \ref{sec: hom harm}). On the contrary, in Proposition \ref{prop: homog} we show the existence of infinitely many geometrically distinct $2$-homogeneous profiles - hence profiles with maximal vanishing order at $0$, not obtainable from one another by rigid motions - satisfying all the extremality conditions identified for blow-ups, related to equation \eqref{2pop}. This already marks a first difference. Moreover, for these homogeneous profiles, the free boundary need not consist of smooth arcs meeting at singular points with equal angles. Such a conclusion would follow only if $\lambda_{2,1}(\omega_i)=\bar \lambda$ were independent of $i$, a non-trivial property which is currently open. As far as we know, the analogous question is open also for optimal partitions of \eqref{spec part prob}.

$2$) Still in view of the parallel with Problem 1, it is natural to ask whether an optimal partition is connected, in the sense that each $\omega_i$ is connected. In the spectral case, this property follows from a simple argument. Indeed, if $A_1, A_2$ are two disjoint open sets, then $\lambda_1(A_1\cup A_2)=\min\{\lambda_1(A_1),\lambda_1(A_2)\}$.
Thus, if $(\omega_1,\dots,\omega_k)$ is an optimal spectral partition and, for instance, $\omega_1=A_1\cup A_2$ with $A_1$ and $A_2$ disjoint and open, then one may replace $\omega_1$ by the component realizing the minimum above without increasing the energy. This yields another optimal partition in which one phase has an empty open component. Such a configuration is incompatible with the unique continuation principle, and hence optimal spectral partitions must be connected.

In the torsional case of Problem A, the situation is substantially different. The preceding argument relies crucially on the minimum property of $\lambda_1$, which fails for the nonlinear eigenvalue $\lambda_{2,1}$. Indeed, as already observed in \cite{BrFr2}, for two disjoint open sets $A_1, A_2$ one has
\be\label{e:conn}
\lambda_{2,1}(A_1\cup A_2) = \left(\frac{1}{\lambda_{2,1}(A_1)} + \frac{1}{\lambda_{2,1}(A_2)}\right)^{-1}.
\ee
This formula allows for a compensation between different connected components, and the connectedness of optimal torsional partitions appears to be a much more delicate issue. We leave this question open. We simply note that, in dimension one, where explicit computations are available, one can verify directly that every optimal partition is connected.

\end{remark}

\subsection{Main results for Problem B} Problem \eqref{pbB} can be treated in a similar way to Problem A. Here we summarize the main results in a single statement where, again, we focus on $N \ge 2$.

\begin{Theorem}
Under the general assumptions of Problem B, the following statements hold true.
\begin{enumerate}
    \item [($i$)] The minimization problem \eqref{pbB} has a solution. 
    \item [($ii$)] Any minimizer is Lipschitz continuous in $\overline{\Omega}$.
    \item [($iii$)] The following extremality conditions are satisfied: 
    \[
-\Delta u_i \leq \lambda_i \ind_{\{u_i>0\}},\quad 
  -\Delta \bigg( u_{i}-\sum_{j\neq i}u_{j}\bigg) \ge \lambda_i \ind_{\{u_{i}>0\}\cup\{\u=0\}} - \sum_{j\neq i}\lambda_j \ind_{\{u_{j}>0\}}\quad \text{in $\O$}.
\]
    \item [($iv$)] The conclusions of both Theorems \ref{prop: non-deg} and \ref{thm: fb} hold for any minimizer of \eqref{pbB}.
\end{enumerate}
\end{Theorem}

We shall not provide detailed proofs of the statements, as they follow by straightforward adaptations of the arguments developed for Problem A. Actually, the constructions of explicit competitors become considerably simpler in the setting of Problem B, where no normalization condition has to be handled. Moreover, no counterpart of Proposition \ref{p:almost} is needed, since the functional is already unconstrained.

In light of the simplified structure of the problem, we can also point out some further results. First, if the parameters  $\lambda_i$ are not all equal, then the construction in Proposition \ref{prop: homog} actually provides limiting problems for which the free boundary consists of smooth arcs meeting at singular points \emph{with different angles}. 

Second, in a specific setting, we can address the question of connectedness of optimal domains discussed in Remark \ref{rem: further} above.

\begin{proposition}\label{prop: connect}
    Under the general assumptions of Problem B, let $\lambda_1 =\cdots=\lambda_k= \bar \lambda>0$, and suppose that $\{\varphi_i>0\}\subset \pa \Omega$ is connected for every $i$. Then, for any minimizer $\u$ of \eqref{pbB}, the positivity set $\{u_i>0\}$ is connected for every $i$.  
\end{proposition}

The proof of this fact is more involved than that of its linear counterpart in Problem 2. 
This reflects the greater delicacy of the connectedness issue in the setting considered here. At the same time, the result seems to suggest that the connectedness property may also hold for the partition problem \eqref{eq:partition.min}.

\begin{remark}
    When dealing with Problem B, we have always supposed that the prescribed traces $\varphi_i$ are non-trivial. One may wonder what happens if one (or some) of the functions $\varphi_i$ vanishes. For instance, one may be tempted to address
    \[
    \inf\left\{
 \sum_{i=1}^k \int_{\Omega} \left( |\nabla u_i|^2 - 2 u_i \right)\,dx
 \ \left|\
 \begin{array}{l}
 u_i \in H^1_0(\Omega),\\
 u_i \cdot u_j \equiv 0 \quad \text{in } \Omega \text{ for all } i\neq j
 \end{array}
 \right.
 \right\},
 \]
 which, arguing as in Theorem \ref{t:existence-equivalence}, is equivalent to 
  \[
  \sup\left\{\sum_{i=1}^k T(\omega_i) \ \left|\
 \begin{array}{l}
 (\omega_1,\dots,\omega_k) \ \text{with }\omega_i \subset \Omega \text{ open and non-empty},\\
 \omega_i \cap \omega_j = \emptyset \quad \text{for all } i \neq j
 \end{array}
 \right.
 \right\}.
    \]
In this case, using the same argument as in Proposition \ref{prop: connect}, it is not difficult to check that any minimizer is of type $(w,0,\dots,0)$, where $w$ is the torsion function in $\Omega$. This shows that, when $\varphi_i =0$, the corresponding density may vanish as well. 
\end{remark}

\subsection{Related problems and future perspectives}

We conclude this introduction by mentioning some related questions and possible directions for future research.

\subsubsection{Optimal partitions for general nonlinear eigenvalues.}
It is natural to introduce a family of variational problems interpolating between the spectral and the torsional settings considered above. More precisely, given an open set $\omega\subset \Omega$ and $q\in[1,2]$, one may consider
\[
 \lambda_{2,q}(\omega)
 :=
 \inf\left\{
 \int_{\omega} |\nabla u|^2\,dx
 \ \left| \
 u \in H^1_0(\omega), \ \|u\|_{L^q(\omega)} = 1
 \right.\right\}.
\]
Note that $\lambda_{2,2}$ coincides with the first Dirichlet eigenvalue $\lambda_1$, while $\lambda_{2,1}$ is defined as in Problem A. We refer to \cite{BrDPFr, BrFr, BrFr2, BrFrRu} for a background on $\lambda_{2,q}$. Then it is natural to study the partition problem
\[
 \mathcal{L}_q:=\inf\left\{
 \sum_{i=1}^k \lambda_{2,q}(\omega_i) \ \left|\
 \begin{array}{l}
 \omega_i \subset \Omega \quad \text{open and non-empty},\\
 \omega_i \cap \omega_j = \emptyset \quad \text{for all } i \neq j
 \end{array}
 \right.
 \right\},
\]
which for $q=2$ is Problem 1, while for $q=1$ is Problem A. What about the intermediate regime? Minimizers should be related to segregated solutions of sublinear equations of the form \eqref{sube} and the corresponding nodal analysis should be connected with the sublinear free boundary problems studied in \cite{SoaTer18}. By combining the strategy developed here with the one in \cite{SoaTer18} (which concerns the full range $q \in [1,2)$) we believe that the main regularity results we proved for $q=1$ could be generalized throughout the whole range $q\in[1,2)$. A further interesting issue is to understand whether the qualitative behavior of optimal partitions changes continuously as $q\to 2^-$, or whether genuinely new phenomena appear in the transition from the linear to the sublinear regime. For instance, since a nonlinear counterpart of \eqref{e:conn} holds for every $q \in [1,2)$, one may ask for which values of $q\in[1,2]$ optimal partitions associated with $\lambda_{2,q}$ have connected phases.

A more sophisticated problem consists in minimizing 
\[
 \mathcal{L}_{q_1,\dots,q_k}:=\inf\left\{
 \sum_{i=1}^k \lambda_{2,q_i}(\omega_i) \ \left|\
 \begin{array}{l}
 \omega_i \subset \Omega \quad \text{open and non-empty}, \ \omega_i \cap \omega_j = \emptyset \quad \text{for all } i \neq j,\\
 q_i \in [1,2] \quad \forall i=1,\dots,k
 \end{array}
 \right.
 \right\}.
\]
Heuristically, in this case the relevant scalar model (at least around regular points) should be 
\[
-\Delta w=\lambda_+(w^+)^{q-1} - \lambda_-(w^-)^{p-1}, \quad \text{with }p \neq q, \ p,q \in [1,2].
\]
This equation has no natural scaling, and nodal properties for local minimizers have been studied in \cite{SoaTor24}. An extension of our results to this inhomogeneous setting seems more challenging.

\subsubsection{Asymptotics and optimal configurations as $k\to+\infty$.}
Finally, one may study the behavior of optimal torsional partitions as the number of phases $k$ tends to infinity. For spectral partitions, this asymptotic regime is connected with the so-called \emph{Honeycomb conjecture} proposed in \cite{CafLin2}. It would be natural to ask whether analogous asymptotic patterns arise for torsional partitions, and whether the limiting cell shapes are governed by the same geometric principles or by different ones. 

\subsubsection{Further variants} Many variants of Problems A and B can be obtained by comparison with the corresponding literature on Problems 1 and 2. For instance, we believe that long-range torsional partitions and segregation problems, in the spirit of \cite{CaPaQu, SoTaZi23, SoTaTeZi18}; torsional partitions with measure constraints, in the spirit of \cite{MaTa}; and torsional segregation phenomena with possible overlap, in the spirit of \cite{SoTe241, SoTe242}, are problems of potential interest.

\subsection{Structure of the paper} In Section \ref{sec: basics}, we prove the existence of a solution to \eqref{e:partition.functions}, establish the optimal regularity of nonlinear eigenfunctions, and consequently complete the proof of Theorem \ref{t:existence-equivalence}. We also derive the extremality conditions for minimizers, namely Proposition \ref{p:classeS} and a domain variation formula. Section \ref{sec: non-deg} is devoted to the proof of the unique continuation principle and of the non-degeneracy property in Theorem \ref{prop: non-deg}. In Section \ref{sec: almost}, we prove the unconstrained almost-minimality property satisfied by minimizers of \eqref{e:partition.functions}, namely Proposition \ref{p:almost}. In Sections \ref{sec: mon} and \ref{sec: upper} we collect the monotonicity formulae and establish an upper semicontinuity property for the vanishing-order map; these results are key ingredients in the blow-up analysis carried out in Section \ref{sec: blow-up}. Section \ref{sec: dim} is devoted to the proof of the Hausdorff dimension estimates for the nodal and singular sets and to the proof of the regularity of the regular set, thus completing the proof of Theorem \ref{thm: fb}. In Section \ref{sec: hom}, we discuss the existence and properties of global homogeneous profiles in dimension $N=2$. Finally, Section \ref{sec: conn} contains the proof of Proposition \ref{prop: connect}.

\subsection{Notations}

We use the standard notation $B_r(x_0)$ (resp. $\partial B_r(x_0)$) for the ball (resp. the sphere) centered at point $x_0$, with radius $r>0$; when $x_0=0$, we often write $B_r=B_r(0)$ and $\partial B_r=\partial B_r(0)$. For any set $A\subset \R^N$, we denote by $\ind_{A}$ the characteristic function of $A$.

We adopt the following vector notation for $k$ valued functions:
\[
\mf{u}=(u_1,\dots,u_k), \quad |\u|^2 = \sum_{i=1}^k u_i^2, \quad |\nabla \u|^2 = \sum_{i=1}^k |\nabla u_i|^2.
\]
Moreover, we say that a vector $\lambda \in \R^k$ is either non-negative or positive if all its components are  either non-negative or positive. 
The positive and negative parts of $u$ are denoted by $u^+$ and $u^-$, respectively.

\section{Existence, optimal regularity and Euler-Lagrange equations}\label{sec: basics}

We begin this section by presenting the proof of the existence of solutions for \eqref{eq:partition.min}.

\begin{proof}[Proof of Theorem \ref{t:existence-equivalence}]
The first point is straightforward: let $\{\mf{u}_n\}$ be a minimizing sequence for $\ell$ (see \eqref{e:partition.functions}). Since we minimize a coercive functional in a weakly closed subset of $H_0^1(\O,\Sigma_k)$ (which is reflexive), the direct method of the calculus of variations gives the existence of a minimizer. The fact that the set of functions $\u \in H^1_0(\O,\Sigma_k)$ such that $\norm{u_i}{L^1(\O)} = 1$ for every $i$ is weakly closed, follows by the compactness of the embedding $H^1(\Omega) \hookrightarrow L^1(\Omega)$.

The second point is postponed to the following Subsection \ref{sub: Lip}. 

Regarding the third point, let $(\omega_1,\dots,\omega_k) \in \mathcal{P}_k(\Omega)$, and for every $i=1,\dots,k$ let $v_i$ be a minimizer for $\lambda_{2,1}(\omega_i)$ (such a minimizer exists, being $\omega_i$ open and bounded). Then
\[
\sum_{i=1}^k \lambda_{2,1}(\omega_i) =  \int_{\O} |\nabla \mf{v}|^2\,dx \ge \ell, 
\]
whence it follows that $\mathcal{L} \ge \ell$, where $\mathcal{L}$ is defined in \eqref{eq:partition.min}. On the other hand, let $\bar\u$ be a minimizer for $\ell$, and let $\bar \omega_i :=\{\bar u_i>0\}$. By point ($ii$), $\bar \omega_i$ is open, and hence
\[
\ell = \int_{\Omega} |\nabla \bar \u|^2\,dx = \sum_{i=1}^k \lambda_{2,1}(\bar\omega_i) \ge \mathcal{L}.
\]
Therefore, $\mathcal{L}=\ell$ and all the previous inequalities are, in fact, equalities. This easily implies also the rest of the thesis.
\end{proof}

\subsection{Lipschitz regularity}\label{sub: Lip}
In order to prove a Lipschitz estimate for minimizers of \eqref{e:partition.functions}, we first need to deduce the Euler-Lagrange equations stated in Proposition \ref{p:classeS}. 

\begin{proof}[Proof of Proposition \ref{p:classeS}]
The proof is inspired by \cite[Theorem 5.1]{ConTerVer05}. We fix $i=1,\dots,k$, and prove the two inequalities separately.

\emph{Step 1)} Let $t >0$, $\varphi \in C^\infty_c(\O)$ be non-negative and define $\mathbf{v} \in H^1(\O,\Sigma_k)$ by
\[
v_i := \frac{(u_i - t \varphi)^+}{\norm{(u_i - t \varphi)^+}{L^1(\O)}}, \qquad v_j := u_j \quad\text{for $j\neq i$}.
\]
Since $\{u_i-t\varphi>0\} \subset \{u_i>0\}$, this is an admissible competitor in the minimality condition, whence we deduce that
\[
\left(\int_\O (u_i-t\varphi)^+\,dx\right)^2 \int_\O |\nabla u_i|^2\,dx \leq \int_\O |\nabla (u_i-t\varphi)^+|^2\,dx .
\]
Now, as $t \to 0^+$ we have that
\[
\int_\O (u_i-t\varphi)^+\,dx = 1 -t \int_{\O}\varphi\ind_{\{u_i>0\}}\,dx + o(t),
\]
where we used the fact that $\|u_i\|_{L^1(\Omega)}=1$, and 
\begin{equation}\label{1051}
    \int_{\O} |\nabla(u_i-t\varphi)^+|^2\,dx = \lambda_{2,1}(\omega_i) - 2t \int_{\O} \nabla u_i \cdot \nabla \varphi\,dx + o(t).
\end{equation}
Therefore
\begin{align*}
\left(1-2 t\int_{\O}\varphi\ind_{\{u_i>0\}}\,dx\right)\lambda_{2,1}(\omega_i)
&\leq \lambda_{2,1}(\omega_i)
-2t\int_{\Omega} \nabla u_i \cdot \nabla \varphi \,dx + o(t),
\end{align*}
as $t \to 0^+$, whence 
\[
\int_{\Omega} \left(\nabla u_i \cdot \nabla \varphi -\varphi \lambda_{2,1}(\omega_i)\ind_{\{u_i>0\}}\right)\,dx \leq 0.
\]
Since this inequality holds for every non-negative $\varphi \in C^\infty_c(\O)$, the first inequality in the thesis follows.

\emph{Step 2)} Let $t >0$, $\varphi \in C^\infty_c(\O)$ be non-negative and set
\[
\hat u_i := u_i - \sum_{j \neq i}u_j.
\]
Now, define $\mathbf{v} \in H^1(\O,\Sigma_k)$ by
\[
v_i := \frac{(\hat u_i + t \varphi)^+}{m_i}, \qquad v_j := \frac{(\hat u_i + t \varphi)^-}{m_j} \ind_{\{u_j>0\} } ,\quad\text{for $j\neq i$},
\]
where
\[
m_i:= \norm{(\hat u_i + t \varphi)^+}{L^1(\O)}, \qquad m_j:=\norm{(\hat u_i + t \varphi)^-}{L^1(\omega_j)}, \quad \text{for $j\neq i$},
\]
so that
\begin{equation}\label{1151}
m_i = 1+t \int_{\O} \varphi \ind_{\{\hat u_i \ge 0\}}\,dx+o(t),\qquad m_j = 1 - t \int_{\O}\varphi \ind_{\{u_j>0\}}\,dx + o(t),\quad\text{for $j\neq i$}.
\end{equation}
Since 
\[
\mathcal{L} = \sum_{i=1}^k \lambda_{2,1}(\omega_i)=\int_{\O} |\nabla \u|^2 \,dx = \int_{\O} |\nabla \hat{u}_i|^2 \,dx,
\]
and
\[
\int_{\O} |\nabla (\hat{u}_i+t\varphi)^+|^2 \,dx = 
\int_{\O} |\nabla (\hat{u}_i+t\varphi)|^2 \,dx - \sum_{j\neq i}\int_{\omega_j} |\nabla (u_j-t\varphi)^+|^2 \,dx,
\]
by considering $\mf{v}$ in the minimality condition we obtain
\begin{align*}
\mathcal{L} = \int_\O |\nabla \hat{u}_i|^2 \,dx & \leq \frac{1}{m_i^2} \int_{\O} |\nabla (\hat u_i+t\varphi)^+|^2\,dx + \sum_{j \neq i} \frac{1}{m_j^2} \int_{\omega_j} |\nabla (\hat u_i+t \varphi)^-|^2\,dx 
\\
&= \frac{1}{m_i^2}\int_\O |\nabla (\hat{u}_i+t\varphi)|^2 \,dx + \sum_{j\neq i}\left(\frac{m_i^2-m_j^2}{m_j^2m_i^2}\right)\int_{\omega_j} |\nabla (u_j-t\varphi)^+|^2 \,dx.
\end{align*}
Recalling \eqref{1051}, this can be rewritten as
\[
m_i^2 \mathcal{L} \leq \mathcal{L} + 2t \int_\O \nabla \hat{u}_i\cdot \nabla \varphi \,dx +  \sum_{j\neq i}\left(\frac{m_i^2-m_j^2}{m_j^2}\right)\left(\lambda_{2,1}(\omega_j) - 2t \int_\O \nabla u_j \cdot \nabla \varphi \,dx \right) + o(t),
\]
as $t \to 0^+$; here we also used the fact that $(m_i^2-m_j^2)/(m_i^2 m_j^2)$ remains bounded for small $t$, since clearly $m_i^2 m_j^2 \to 1$ as $t \to 0^+$ (by \eqref{1151}), and 
\[
m_i^2-m_j^2= 2 t \int_\O \varphi \left(\ind_{\{\hat u_i\ge 0\}}+\ind_{\{u_j>0\}}\right)\,dx + o(t).
\]
Therefore, dividing by $2t$ and passing to the limit as $t \to 0^+$, we infer that
\begin{align*}
0\leq & \int_\O \left(\nabla \hat{u}_i \cdot \nabla \varphi - \mathcal{L} \varphi \ind_{\{\hat u_i\ge 0\}}\right)\,dx + \sum_{j \neq i} \lambda_{2,1}(\omega_j) \int_{\O} \varphi(\ind_{\{\hat u_i\ge 0\}}+\ind_{\{u_j>0\}})\,dx \\
=&\int_\O \nabla \hat{u}_i \cdot \nabla \varphi \,dx -\lambda_{2,1}(\omega_i)\int_\O \varphi \ind_{\{\hat u_i\ge 0\}}\,dx  + \sum_{j\neq i}\lambda_{2,1}(\omega_j)\int_{\O}\varphi \ind_{\{u_j>0\}}\,dx.
\end{align*}
Since $\varphi \ge 0$ was arbitrarily chosen, the thesis follows.
\end{proof}

Next, we recall the following Lipschitz estimate for a class of segregated configurations (see \cite[Theorem 8.3]{ConTerVer05}) .
\begin{Lemma}\label{l:lipschitz-conti}
Let $M>0$ and $\u \in \mathcal{S}^*_{M,k}(\O)$, where 
\[
\mathcal{S}^*_{M,k}(\O):= \left\{\u \in H^1_0(\O,\Sigma_k) \ \left| \  u_i \geq 0,\quad -\Delta u_i \leq M,\quad 
  -\Delta \left( u_{i}-\sum_{j\neq i}u_{j}\right) \ge -M\quad \text{in $\O$}\right.\right\}.
\]
Then $\u$ is locally Lipschitz continuous in $\O$, and there exists a dimensional constant $C>0$, such that  
    \[
    \sup_{i=1,\dots,k}[u_i]_{C^{0,1}(\overline{B_{r/2}(x_0)})} \leq \frac{C}{r} \norm{\u}{x_0,r},
    \] 
    for every $x_0 \in \O$ and  $r\in (0,\mathrm{dist}(x_0,\partial\O))$. 
\end{Lemma}
The previous Lipschitz estimate is essentially deduced in \cite{ConTerVer05} by combining the Caffarelli-Jerison-Kenig monotonicity formula in \cite{CafJerKenACF} with Campanato's embedding (see also Theorem 1.3 and Remark 1.5 in \cite{CafJerKenACF}).

\begin{Theorem}[Lipschitz continuity of minimizers]\label{t:Lipschitz}
Let $\O$ be a domain satisfying the uniform exterior sphere condition and $\u \in H^1_0(\O,\Sigma_k)$ be a minimizer of \eqref{e:partition.functions}. 

Then $\u \in C^{0,1}(\overline{\O},\Sigma_k)$ and there exists $C>0$, depending only on the dimension and $\O$, such that  
\begin{equation}\label{eq:lip:estimates}
        \sup_{i=1,\dots,k}[u_i]_{C^{0,1}(\overline{ B_{r/2}(x_0)})} \leq \frac{C}{r} 
        \|\u\|_{x_0,r},   
\end{equation}
    for every $x_0 \in \O$ and  $r\in (0,\mathrm{dist}(x_0,\partial\O))$. 
\end{Theorem}

\begin{proof}
    The Lipschitz continuity of $\u$ in the interior of $\O$ follows immediately by Lemma \ref{l:lipschitz-conti}. Indeed, by Proposition \ref{p:classeS}, if we set 
    \[
    \overline{\Lambda} := \max_{i=1,\dots,k}\lambda_{2,1}(\omega_i),
    \]
    it is immediate to see that minimizers belong to the class of segregated functions $\mathcal{S}_{\overline{\Lambda},k}^*(\O)$. Regarding the regularity up to the boundary, the estimate follows from a comparison argument. Let $w \in H^1_0(\O)$ be the unique solution to
   \[
   -\Delta w = \overline{\Lambda} \quad\text{in }\O,\qquad w=0\quad\text{on }\partial \O.
   \]
   By the comparison principle $0\leq u_i\leq w$ in $\O$, for every $i=1,\dots,k$, and $w \in C^{0,1}(\overline{\O})$ by the uniform exterior sphere condition. Therefore
\[
0\leq u_i \leq [w]_{C^{0,1}(\overline{\O})} \mathrm{dist}(x,\partial \O) \quad \text{in }\Omega.
\]
At this point, the Lipschitz continuity up to the boundary is a rather standard fact, for which we refer to \cite[Proposition 3.3]{HanLin} or \cite[Theorem 3.4]{SoTaTeZi18}.
\end{proof}

Having established that the minimizers of \eqref{e:partition.functions} are Lipschitz continuous in $\overline\O$, the proof of Theorem \ref{t:existence-equivalence} is now complete.

Next, we deduce two immediate consequences of this property. On one hand, by exploiting the continuity of each component of the minimizers, we introduce a different formulation of the PDEs in Proposition \ref{p:classeS}.

\begin{Corollary}\label{cor:EL eq}
Let $\u \in H^1_0(\O,\Sigma_k)$ be a minimizer of \eqref{e:partition.functions} and $i=1,\dots,k$. Then, there exists a non-negative Radon measure $\mu_i \in (C_c(\O))'$ supported on $\pa \{u_i>0\}$, such that 
\[
-\Delta u_i = \lambda_{2,1}(\omega_i)\ind_{\{u_i>0\}} - \mu_i \quad\text{in $\O$,}
\]
where the equation holds in the sense of distributions.
\end{Corollary}
\begin{proof}
It is a consequence of Proposition \ref{p:classeS}, the fact that $\{u_i>0\}$ is open (and hence $u_i$ satisfies $-\Delta u_i=\lambda_{2,1}(\omega_i)$ in $\omega_i$), and the fact that any non-negative distribution is a non-negative Radon measure (we refer e.g. to \cite[Chapter 2]{Hor}). Alternatively, one can argue more directly as in \cite[Lemma 5.5]{TaTe12}.
\end{proof}
On the other hand, we can also derive the Euler-Lagrange equation related to inner variations. In particular, we deduce the validity of a domain variation formula (or Pohozaev-type identity) for minimizers of problem \eqref{e:partition.functions}.
\begin{Lemma}\label{l:inner}
If $\u \in H^1_0(\O,\Sigma_k)$ is a minimizer of \eqref{e:partition.functions}, then
\[
    \sum_{i=1}^k \int_{\O}\Big( \left(|\nabla u_i|^2 -  2 \lambda_{2,1}(\omega_i)u_i\right)\mathrm{div}\xi  - 2 D\xi\cdot\nabla u_i\cdot  \nabla u_i \Big)\,dx = 0,
\]
for every $\xi \in C^\infty_c(\O,\R^N)$. In particular, for every $x_0\in \{\u=0\}$ and $r \in (0,\dist(x_0,\pa \Omega))$ we have that
\[
\begin{split}
\frac{N-2}{2}\int_{B_r(x_0)} &|\nabla \u|^2\,dx   + r\int_{\partial B_r(x_0)} |\partial_\nu \u|^2\,d\sigma\\
&= \frac{r}{2}\int_{\partial B_r(x_0)} |\nabla \u|^2\,d\sigma 
+N \sum_{i=1}^k \lambda_{2,1}(\omega_i) \int_{B_r(x_0)}  u_i  \,dx 
-r \sum_{i=1}^k\lambda_{2,1}(\omega_i)\int_{\partial B_r(x_0)} u_i\,d\sigma.
\end{split}
\]
\end{Lemma}
\begin{proof}
Let $\xi \in C^\infty_c(\R^N,\R^N)$ be a given vector field with compact support in $\O$, and let $\Psi_t : \O\to \O$ be the diffeomorphism
\[
\Psi_t(x) := x + t \xi(x) \quad\text{for }x \in \O.
\]
Therefore, for $t$ small enough, the inverse map $\Phi_t := (\Psi_t)^{-1}$ satisfies
\[
\Phi_t (x) := x - t \xi(x) + o(t) \quad\text{for $x\in\O$ as $t \to 0^+,$}
\]
where the remainder term $o(t)$ should be understood in the $C^1(\overline{\Omega})$ sense. Thus, the vector-valued function $\u_t := (u_{t,1},\dots,u_{t,k})$ with
\[
u_{t,i} := \frac{u_{i} \circ \Phi_t}{\norm{u_{i} \circ \Phi_t}{L^1(\O)}} \in H^1_0(\O,\Sigma_k),
\]
is well-defined and is an admissible competitor for the minimization problem. For every $i=1,\dots,k$ we have
\begin{align*}
\int_{\O}|\nabla (u_i \circ \Phi_t)|^2\,dx &= \lambda_{2,1}(\omega_i) + t \int_\O \Big(|\nabla u_i|^2 \mathrm{div}\xi -2 D \xi \nabla u_i \cdot \nabla u_i \Big)\,dx + o(t), \\
\norm{u_{i} \circ \Phi_t}{L^1(\O)}^2  &= 1 + 2t \int_\O u_i \mathrm{div}\xi \,dx + o(t), 
\end{align*}
and hence
\[
\frac{\int_{\O}|\nabla (u_i \circ \Phi_t)|^2\,dx}{\norm{u_{i} \circ \Phi_t}{L^1(\O)}^2}
= \lambda_{2,1}(\omega_i) +t \int_{\O} \Big(\left(|\nabla u_i|^2 -  2 \lambda_{2,1}(\omega_i)u_i\right)\mathrm{div}\xi  - 2 D\xi\cdot\nabla u_i\cdot  \nabla u_i \Big)\,dx + o(t),
\]
where all the remainders $o(t)$ are uniform in $x \in \O$. Moreover, since $\u_0=\u$ is a minimizer, the functional
\[
t \mapsto \sum_{i=1}^k \frac{1}{\norm{u_{i} \circ \Phi_t}{L^1(\O)}^2}\int_{\O}|\nabla (u_i \circ \Phi_t)|^2\,dx,
\]
has a minimum at $t=0$, which leads to the first claimed condition in the thesis.

Now, without loss of generality, we assume that $x_0=0$. For $\eps>0$, choose $\rho_\eps\in C^\infty_c(B_r)$ such that 
    \[
    \rho_\eps =1 \quad\text{in }B_{(1-\eps)r},\qquad 
    \nabla \rho_\eps = -\frac{1}{r\eps} \frac{x}{|x|}+ o(\eps)\quad \text{in }B_{r}\setminus B_{(1-\eps)r}.
    \]
We consider the inner variation along the direction $\xi_\eps(x):= x \rho_\eps(x)$. Since
\[
\mathrm{div}\,\xi_\eps (x) = N \rho_\eps (x) + x\cdot \nabla \rho_\eps(x), \qquad 
D \xi_\eps (x) = \rho_\eps(x) \mathrm{Id} + x \otimes \nabla \rho_\eps(x),
\]
by the previous computations, we deduce that
\begin{align*}
0=&\sum_{i=1}^k \int_{\O} \left(\left(\frac12|\nabla u_i|^2 -\lambda_{2,1}(\omega_i)  u_i\right)\mathrm{div}\xi_\eps  - D\xi_\eps\cdot\nabla u_i\cdot  \nabla u_i \right)\,dx \\
=&\sum_{i=1}^k \int_{\O} \left(\frac{N-2}{2}|\nabla u_i|^2 - N \lambda_{2,1}(\omega_i)  u_i \right)\rho_\eps \,dx\\
&\quad + \sum_{i=1}^k \int_{\O} \left(\left(\frac12|\nabla u_i|^2 -  \lambda_{2,1}(\omega_i)  u_i\right)(x\cdot \nabla \rho_\eps) - (x\cdot \nabla u_i)(\nabla \rho_\eps \cdot \nabla u_i)\right)\,dx \\
=&\sum_{i=1}^k \int_{B_r} \left(\frac{N-2}{2}|\nabla u_i|^2 -N  \lambda_{2,1}(\omega_i) u_i \right)\rho_\eps \,dx\\
&\quad - \sum_{i=1}^k \frac{1}{r\eps}\int_{B_r\setminus B_{(1-\eps)r}} \left(\frac12|\nabla u_i|^2 -  \lambda_{2,1}(\omega_i)  u_i - (\partial_\nu u_i)^2\right)|x|\,dx + o(\eps).
\end{align*}
The thesis follows by passing to the limit as $\eps \to 0^+$.
\end{proof}

\section{Unique continuation property, vanishing order and non-degeneracy}\label{sec: non-deg}

Let $(\omega_1,\dots,\omega_k)$ be an optimal partition of \eqref{eq:partition.min}, and let $\u$ be the associated optimal function for problem \eqref{e:partition.functions}. In this section we prove Theorem \ref{prop: non-deg}, which is the starting point in the study of the nodal set and the free boundary, defined by
\[
Z(\u) = \{x \in \Omega \ \left| \ \u(x) = 0 \right.\} = \Omega \setminus \bigcup_{i=1}^k \omega_i \quad \text{and} \quad
\Gamma(\u)= \bigcup_{i=1}^k \partial \{u_i>0\} \cap \Omega = \bigcup_{i=1}^k \partial \omega_i \cap \Omega,
\]
respectively. We recall the notions of vanishing order and non-degeneracy introduced in Definition \ref{D:VOi} and Theorem \ref{prop: non-deg}, respectively. 

\begin{proof}[Proof of Theorem \ref{prop: non-deg}]
Suppose by contradiction that there exist $x_0\in Z(\u)$ and  $r_n\to 0^+$ such that 
    \begin{equation}\label{eq:nd:proof:1}
         \lim_{ n} \frac{\|\u\|_{x_0,2r_n}}{r_n^{2}}=0.
    \end{equation}
    Note that also
    \be\label{eq:nd:proof:1'}
    \lim_n \frac{1}{r_n^2}\left(\frac{1}{r_n^N} \int_{B_{2r_n}(x_0)} |\u|^2\,dx\right)^{1/2}=0,
    \ee
    by the Poincar\'e inequality.
    
    Let $\underline {\Lambda}:= \min_{i=1,\dots,k} \lambda_{2,1}(\omega_i)>0$, and let $w \in H^1_0(B_1)$ be the unique solution to 
    \[
    -\Delta w = \underline {\Lambda} \quad \text{in $B_1$},\qquad w=0\quad \text{on $\partial B_1$}.
    \]
    Then, define 
    \[
    w_n:=\begin{cases}
      \displaystyle  r_n^{2} w\Big(\frac{x-x_0}{r_n}\Big) &\text{in }B_{r_n}(x_0)\\
        0&\text{in }\Omega\setminus B_{r_n}(x_0).
    \end{cases}
    \]
    which satisfies   \begin{equation}\label{eq:nd:proof:2}
            -\Delta w_n \le \underline {\Lambda} \ind_{\{w_n>0\}}\qquad \text{in } \Omega, 
    \end{equation}
    and
    \begin{equation}\label{eq:nd:proof:3}
        \int_{\Omega} |\nabla w_n|^2\, dx = \underline {\Lambda} r_n^{N+2}\int_{B_1}w\,dx.
    \end{equation}
Finally, let $\eta_n\in C_c^\infty(\Omega)$ be such that $\eta_n=0$ in $\Omega\setminus B_{2{r_n}}(x_0)$, $\eta_n=1$ in $B_{r_n}(x_0)$ and $|\D \eta_n|\le C/{r_n}$ for some dimensional constant $C>0$. For every $i=1,\dots,k$, by testing 
\[
-\Delta u_i \leq \lambda_{2,1}(\omega_i)\ind_{\{u_i>0\}},
\]
with $\eta_n^2 u_i$ and using standard computations combined with \eqref{eq:nd:proof:1'}, we obtain that
\begin{equation}\label{eq:nd:proof:4}
    \int_{B_{2{r_n}}(x_0)}\eta_n^2|\D u_i|^2\,dx \le C \int_{B_{2{r_n}}(x_0)} |\D \eta_n|^2 u_i^2\,dx + C \int_{B_{2{r_n}}(x_0)} \eta_n^2 u_i \, dx = o(r_n^{N+2}).
\end{equation}
Moreover, testing \eqref{eq:nd:proof:2} with $u_i$ and using again \eqref{eq:nd:proof:1'}, 
we also have
\begin{equation}\label{eq:nd:proof:5}
\int_{\Omega} \D w_n\cdot \D u_i\,dx \le \underline {\Lambda}\int_{B_{r_n}(x_0)} u_i\,dx =o(r_n^{N+2}).
\end{equation}
Furthermore, by \eqref{eq:nd:proof:1} and \eqref{eq:nd:proof:1'}
\be\label{eq:nd:proof:6}
\begin{split}
\int_{\Omega} (1-\eta_n)&u_j \D \eta_n\cdot\D u_j\,dx  \le \frac{C}{r_n}\int_{B_{2r_n}(x_0)} |u_j||\nabla u_j| \,dx \\
& \le \frac{C}{r_n} \cdot r_n^{N/2}\left( \frac{1}{r_n^N} \int_{B_{2r_n}(x_0)} u_j^2\,dx\right)^{1/2} \cdot r_n^{\frac{N-2}{2}} \left( \frac{1}{r_n^{N-2}} \int_{B_{2r_n}(x_0)} |\nabla u_j|^2\,dx\right)^{1/2}
= o(r_n^{N+2}).
\end{split}
\ee
Now, let $\mathbf{v}_n=(v_{1,n},\dots,v_{k,n})$ be defined by
\[
v_{1,n}:= \frac{u_1+w_n}{\norm{u_1+w_n}{L^1(\O)}},\qquad v_{j,n}:=\frac{(1-\eta_n)u_j}{\norm{(1-\eta_n)u_j}{L^1(\O)}}, \text{ for }j=2,\dots,k\,,
\]
which is an admissible competitor for the minimization problem. Note that for $j=2,\dots,k$ we have
\begin{equation}
\label{1251}\begin{split}
\norm{u_1+w_n}{L^1(\O)} &= 1+\int_{B_{r_n}(x_0)} w_n\,dx = 1+ r_n^{N+2}\int_{B_1}w\,dx,\\
\norm{(1-\eta_n)u_j}{L^1(\O)} & \geq 1- \int_{B_{2r_n}(x_0)}u_j\,dx  \geq 1 + o(r_n^{N+2}).
\end{split}
\end{equation}
Combining the first equality with \eqref{eq:nd:proof:3} and \eqref{eq:nd:proof:5}, we deduce that
\begin{align*}
\int_\O |\nabla v_{1,n}|^2\,dx &= \frac{1}{\norm{u_1+w_n}{L^1(\O)}^2}\int_\O |\nabla (u_1+ w_n)|^2\,dx\\
& = \left(1- 2 r_n^{N+2}\int_{B_1}w\,dx \right)\left(\lambda_{2,1}(\omega_1) + \underline {\Lambda} r_n^{N+2}\int_{B_1} w  \,dx\right) + o(r_n^{N+2})\\
&= \lambda_{2,1}(\omega_1) +  r_n^{N+2} (\underline {\Lambda} -2 \lambda_{2,1}(\omega_1))\int_{B_1}w\,dx + o(r_n^{N+2});
\end{align*}
in addition, the second inequality in \eqref{1251} together with \eqref{eq:nd:proof:4} and 
\eqref{eq:nd:proof:6} gives
\begin{align*}
\int_\O |\nabla v_{j,n}|^2\,dx &= \frac{1}{\norm{(1-\eta_n)u_j}{L^1(\O)}^2}\int_\O \Big(|\nabla \eta_n|^2 u_j^2 + (1-\eta_n)^2 |\nabla u_j|^2 - 2(1-\eta_n)u_j\nabla \eta_n \cdot \nabla u_j\Big)\,dx\\
& \le \frac{1}{1+o(r_n^{N+2})}  \left(\int_{\O} |\nabla u_j|^2\,dx + o(r_n^{N+2})\right) = \lambda_{2,1}(\omega_j) +o(r_n^{N+2}).
\end{align*}
Then, using the previous estimates and the minimality of $\u$, we deduce that
\[
    0\le \sum_{i=1}^k\left(\int_\O |\nabla v_{i,n}|^2 \,dx - \lambda_{2,1}(\omega_i) \right)
    \le  r_n^{N+2}(\underline {\Lambda} -2\lambda_{2,1}(\omega_1))\int_{B_1} w\,dx + o(r_n^{N+2}),
\]
which leads to a contradiction for large $n$, since $w>0$ in $B_1$ and $\underline {\Lambda} < 2\lambda_{2,1}(\omega_i)$ for every $i=1,\dots,k$. The non-degeneracy directly implies that $\cO(\u,x_0) \le 2$ for any $x_0 \in Z(\u)$. The fact that $\cO(\u,x_0) \ge 1$ follows by the Lipschitz continuity. The rest of the thesis follows straightforwardly.
\end{proof}

\section{Almost-minimality condition}\label{sec: almost}
In order to better understand the nature of the optimal partition, it is convenient to remove the $L^1$-constraint in the minimization problem \eqref{e:partition.functions}. Precisely, let $\Lambda:= (\lambda_{2,1}(\omega_1),\dots,\lambda_{2,1}(\omega_k))$; by exploiting the Lipschitz continuity of minimizers of \eqref{e:partition.functions}, we can show the validity of the (unconstrained) almost-minimality condition close to nodal points stated in Proposition \ref{p:almost}.

\begin{proof}[Proof of Proposition \ref{p:almost}]
Without loss of generality, we may assume that $\mf w$ is non-negative. Otherwise, we can replace it with $\overline{\mf w}:=(|w_1|,\dots,|w_k|) \in H^1(B_r(x_0),\Sigma_k)$, which has the same trace as $\mf w$ on $\partial B_r(x_0)$ since $\u$ is non-negative. Furthermore, because $-(\Lambda \cdot \overline {\mf w}) \le (\Lambda \cdot \mf w)$, if \eqref{e:almost-minimality} holds for $\overline{\mf w}$, then the same inequality also holds for $\mf w$.

For the sake of simplicity, we assume that $x_0=0 \in Z(\u)$. Moreover, we introduce the notation
\[
T_i(g,r):=\int_{B_r}\left(|\nabla g|^2 - 2\lambda_{2,1}(\omega_i)g\right)\,dx.
\]

\emph{Step 1)} We claim that
\be\label{e:almost1}
\sum_{i=1}^k \frac{T_i(u_i,r)}{(1+\delta_i)^2}
\leq 
\sum_{i=1}^k \frac{T_i(w_i,r)}{(1+\delta_i)^2}, \qquad \text{where }\delta_i := \int_{B_r}(w_i-u_i)\,dx, 
\ee 
for all $r>0$ such that $B_r \subset \Omega$ and $\mf{w} \in H^1(B_r)$ with $\mf{w}=\mf{u}$ on $\pa B_r$ and $\mf{w} \ge 0$
(note that $\delta_i>-1$ for $r>0$ sufficiently small; indeed, it is sufficient to take $r$ such that $\supp(u_i) \setminus B_r \neq \emptyset$, so that $\|u_i\|_{L^1(B_r)}<1$). To prove this claim, we start observing that, by minimality,
\be\label{e:mi}
\sum_{i=1}^k \lambda_{2,1}(\omega_i) = \sum_{i=1}^k\int_\O |\nabla u_i|^2\,dx \leq \sum_{i=1}^k \int_\O |\nabla v_i|^2\,dx,
\ee
for every $\mathbf{v}\in H^1_0(\O,\Sigma_k)$ such that $\norm{v_i}{L^1(\O)}=1$ for every $i=1,\dots,k$.
Now, define $\mathbf{v}\in H^1_0(\O,\Sigma_k)$ by
\[
v_i := 
\begin{cases}
{w_i}/{m_i} & \mbox{in }B_r\\
{u_i}/{m_i} & \mbox{in }\O\setminus B_r,
\end{cases}\qquad m_i:= 1+ \delta_i.
\]
This is an admissible competitor in \eqref{e:mi}, and since
\begin{align*}
\int_\O |\nabla v_i|^2 \,dx 
= \frac{1}{m_i^{2}} \left(\lambda_{2,1}(\omega_i) +\int_{B_r}\left(|\nabla w_i|^2 - |\nabla u_i|^2\right) \,dx\right),
\end{align*}
by \eqref{e:mi} we infer that 
\[
\sum_{i=1}^k \lambda_{2,1}(\omega_i) \leq \sum_{i=1}^k\frac{1}{m_i^{2}}\lambda_{2,1}(\omega_i) +\sum_{i=1}^k \frac{1}{m_i^{2}}\int_{B_r}\left(|\nabla w_i|^2 - |\nabla u_i|^2\right) \,dx;
\]
that is (recalling that $m_i=1+\delta_i$ and the definition of $T_i$)
\[
\begin{split}
0 & \leq \sum_{i=1}^k\frac{1}{m_i^2}\left(\int_{B_r}\left(|\nabla w_i|^2 - |\nabla u_i|^2\right) \,dx + \lambda_{2,1}(\omega_i) ( 1- m_i^{2})\right) \\
& = \sum_{i=1}^k\frac{1}{m_i^2}\left(\int_{B_r}\left(|\nabla w_i|^2 - |\nabla u_i|^2\right) \,dx - \lambda_{2,1}(\omega_i) \left( \delta_i^2 +2\int_{B_r} (w_i-u_i)\,dx\right) \right)\\
& = \sum_{i=1}^k\frac{1}{m_i^2}\left(T_i(w_i,r)-T_i(u_i,r) - \lambda_{2,1}(\omega_i)\delta_i^2\right) \le \sum_{i=1}^k\frac{1}{(1+\delta_i)^2}\left(T_i(w_i,r)-T_i(u_i,r)\right),
\end{split}
\]
where we used the fact that $\lambda_{2,1}(\omega_i)>0$ for every $i$. This completes the proof of the claim.

\emph{Step 2)} We claim that for every $\mf{w}$ such that
\be\label{1451}
\sum_{i=1}^k T_i(w_i,r) \le \sum_{i=1}^k T_i(u_i,r),
\ee
the estimates
\be\label{e:qui2}
\int_{B_r}|\nabla (\mathbf{w}-\u)|^2\,dx \leq C \int_{B_r}|\nabla \u|^2\,dx + C r^{N+2} \quad \text{and} \quad |\delta_i| \le C r^{N+1},
\ee
hold for sufficiently small $r>0$.

First, since $w_i-u_i \in H^1_0(B_r)$, by the Poincar\'{e} inequality we have 
\be\label{e:qui}
|\delta_i|\leq C r^{N/2+1}\left(\int_{B_r}|\nabla (w_i-u_i)|^2\,dx\right)^{1/2}.
\ee
On the one hand, we have the following upper bound:
\be\label{1452}
\begin{split}
T_i(w_i-u_i,r) &=  T_i(w_i,r) - T_i(u_i,r) + 2 \int_{B_r}\nabla u_i \cdot \nabla(u_i-w_i)\,dx \\
&\leq T_i(w_i,r) - T_i(u_i,r) +\frac{1}{\eps^2} \int_{B_r}|\nabla u_i|^2\,dx+ \eps^2 \int_{B_r}|\nabla (w_i-u_i)|^2\,dx ,
\end{split}
\ee
which by \eqref{1451} implies that 
\[
\sum_{i=1}^k T_i(w_i-u_i,r)\leq \frac{1}{\eps^2} \int_{B_r}|\nabla \u|^2\,dx+ \eps^2 \int_{B_r}|\nabla (\mathbf{w}-\u)|^2\,dx,
\]
where we choose $\eps>0$ so small that $\eps^2<1/2$. On the other hand, by applying Young's inequality to \eqref{e:qui}, we obtain
\begin{align*}
T_i(w_i-u_i,r) &= \int_{B_r} |\nabla (w_i-u_i)|^2\,dx -2 \lambda_{2,1}(\omega_i)\delta_i\\
&\geq \int_{B_r}|\nabla (w_i-u_i)|^2\,dx - C r^{N+2} - \frac12 \int_{B_r}|\nabla (w_i-u_i)|^2\,dx\\
&= \frac12 \int_{B_r}|\nabla (w_i-u_i)|^2\,dx- C r^{N+2}.
\end{align*}
Thus, summing this estimate through $i=1,\dots,k$ and combining it with the previous upper bound, we get the first inequality in \eqref{e:qui2}; as a consequence, by exploiting the Lipschitz continuity of $\u$ (see Theorem \ref{t:Lipschitz}), we also deduce the second inequality in \eqref{e:qui2}:
\[
\delta_i^2 \leq \sum_{j=1}^k \delta_j^2 \leq C r^{N+2} \left(\int_{B_r}|\nabla \u|^2\,dx + r^{N+2}\right) \leq C r^{2N+2},
\]
as desired.

\emph{Step 3)} The thesis of the proposition follows by showing that, for sufficiently small $r$, the denominators in \eqref{e:almost1} can be removed at the cost of an additive error term. 
In order to show \eqref{e:almost-minimality}, it is not restrictive to assume that \eqref{1451} holds, otherwise \eqref{e:almost-minimality} is trivially satisfied. Therefore, by Step 2 there exists $r_0>0$ so that 
\[
\left|\frac{2+\delta_i}{(1+\delta_i)^2}\right| <\frac52, \qquad \text{for all } r \in (0,r_0),
\]
and hence
\[
\left|\frac{T_i(g,r)}{(1+\delta_i)^2}-T_i(g,r)\right| \le C|\delta_i| |T_i(g,r)|, \qquad \text{for all } g \in H^1(B_r).
\]
By combining the latter estimate with \eqref{e:almost1}, we infer that 
\[
\int_{B_r} \left(|\nabla \u|^2 - 2 (\Lambda \cdot \u)\right)\,dx \leq \int_{B_r} \left(|\nabla \mathbf{w}|^2 - 2 (\Lambda \cdot \mathbf{w})\right)\,dx + C \sum_{i=1}^k |\delta_i|\left(|T_i(w_i,r)|+|T_i(u_i,r)|\right).
\]
It remains to estimate the sum on the right hand side. By Lipschitz continuity
\[
|T_i(u_i,r)| \le \int_{B_r} |\nabla u_i|^2\,dx + 2\lambda_{2,1}(\omega_i) \int_{B_r} u_i \,dx \le C r^N(1+r) \le Cr^N.
\]
Moreover, by \eqref{e:qui2} and the Lipschitz continuity of $\mf{u}$
\begin{align*}
|T_i(w_i-u_i,r)| & \leq  \int_{B_r}|\nabla (w_i-u_i)|^2\,dx + 2\lambda_{2,1}(\omega_i)|\delta_i|\\
&\leq  C\left( \int_{B_r}|\nabla\u|^2\,dx + r^{N+2}\right) + C r^{N+1} \leq  C r^N.
\end{align*}
Thus, in view of \eqref{e:qui2} and \eqref{1452},
\[
|T_i(w_i,r)| \leq |T_i(w_i-u_i,r)| +|T(u_i,r)|  + \int_{B_r}|\nabla u_i|^2\,dx  +  \int_{B_r}|\nabla (w_i-u_i)|^2\,dx \leq C r^N.
\]
By collecting all the previous estimates and combining them with \eqref{e:qui2} again, we finally infer that 
\[
\sum_{i=1}^k |\delta_i| \left(|T_i(w_i,r)|+|T_i(u_i,r)|\right)\leq C r^{2N+1},
\]
which gives the desired result.
\end{proof}

\begin{remark}
    In the following, we will apply the almost minimality condition to the rescaled function 
    \[
\u_r(x):=\frac{1}{\rho_r}\u(x_0+rx),
\]
where $x_0\in Z(\u)$ and $\rho_r\to 0^+$ as $r\to 0^+$ is chosen in such a way that $\rho_r/r^2\ge c$, for some constant $c>0$ which does not depend on $r$. Hence, by rescaling the almost-minimality condition \eqref{e:almost-minimality}, we deduce that there exists $C>0$ such that 
\begin{equation}\label{eq:alm:min:rescaled}
\int_{B_t}\left( |\nabla \u_{r}|^2 - 2\frac{r^2}{\rho_r}(\Lambda\cdot \u_{r})\right)\,dx\leq 
\int_{B_t}\left(|\nabla \mf{v}|^2 - 2\frac{r^2}{\rho_r}(\Lambda \cdot \mf{v})\right)\,dx + C t^{2N+1} \frac{r^4}{\rho_r^2}r^{N-1},
\end{equation}
for any $t>0$, $r>0$ small enough, and any $\mf v\in H^1(B_t,\Sigma_k)$ with $\mf v = \u_r$ on $\partial B_t$. Note that the last term is negligible, being $N\geq 2$ and by the assumption $\rho_r/r^2\ge c>0$.
\end{remark}

\section{Monotonicity formulae and vanishing order}\label{sec: mon}

We collect here some monotonicity formulae which will be used in the next sections in order to study the free-boundary regularity. These formulae must be applied not only to minimizers of problem \eqref{e:partition.functions}, but also to possible rescalings and to limits of rescalings. Therefore, it is convenient to introduce classes of functions that do not impose any boundary conditions on the data.

\begin{Definition}\label{def T}
    Let $\Omega \subset \R^N$ be a domain and $\lambda \in\R^k$ be positive. We say that $\u \in \mathcal{T}_\lambda(\Omega)$ if 
\begin{enumerate}
\item[($i$)] $\u \not \equiv 0$ is non-negative, locally Lipschitz continuous in $\Omega$, and $u_i u_j \equiv 0$ in $\Omega$ for every $i \neq j$;
\item[($ii$)] each component $u_i$ satisfies
\[
-\Delta u_i = \lambda_i \quad \text{in $\{u_i>0\}$};
\]
\item[($iii$)] for every $x_0\in Z(\u)$ and $r \in (0,\dist(x_0,\pa \Omega))$ we have that
\[
\begin{split}
r\int_{\partial B_r(x_0)} |\nabla \u|^2\,d\sigma &=
(N-2)\int_{B_r(x_0)} |\nabla \u|^2\,dx + 2r\int_{\partial B_r(x_0)} |\partial_\nu \u|^2\,d\sigma\\
&\quad +2r \int_{\partial B_r(x_0)} (\lambda \cdot \u)\,d\sigma - 2N\int_{B_r(x_0)} (\lambda\cdot \u) \,dx.
\end{split}
\]
\end{enumerate}
We say that $\u \in \mathcal{T}_{\lambda,\loc}(\Omega)$ if $\u \in \mathcal{T}_\lambda(\Omega')$ for every $\O'\subset\O$.
\end{Definition}
By Theorem \ref{t:Lipschitz}, Corollary \ref{cor:EL eq} and Lemma \ref{l:inner} we already know that if $\u\in H^1_0(\O,\Sigma_k)$ is a minimizer of \eqref{e:partition.functions}, then $\u \in \mathcal{T}_\Lambda(\O)$ with $\Lambda:=(\lambda_{2,1}(\omega_1),\dots,\lambda_{2,1}(\omega_k))$.

\begin{Definition}\label{def G}
Let $\Omega \subset \R^N$ be a domain. We say that $\u \in \mathcal{G}(\Omega)$ if 
\begin{enumerate}
\item[($i$)] $\u \not \equiv 0$ is non-negative, locally Lipschitz continuous in $\Omega$, and $u_i u_j \equiv 0$ in $\Omega$ for every $i \neq j$;
\item[($ii$)] each component $u_i$ is harmonic in $\{u_i>0\}$;
\item[($iii$)] for every $x_0 \in \Omega$ and $r \in (0,\dist(x_0,\pa \Omega))$ 
\[
r\int_{\partial B_r(x_0)} |\D \u|^2 \,d\sigma = (N-2) \int_{B_r(x_0)} |\D \u|^2\,dx 
+2r\int_{\partial B_r(x_0)} |\pa_\nu \u|^2 \,d\sigma.
\]
\end{enumerate}
We say that $\u \in \mathcal{G}_{\loc}(\Omega)$ if $\u \in \mathcal{G}(\Omega')$ for every $\O'\subset\O$.
\end{Definition}

This class of functions has been introduced and studied in \cite{TaTe12}. Inspired by Proposition \ref{p:almost}, we introduce the following functional.

\begin{Definition}\label{def J}
Let $\Omega \subset \R^N$ be a domain, $\lambda \in \R^k$ be non-negative and let
\[
J_{\lambda,\Omega}(\mf{u}):= \int_{\Omega} \Big(|\nabla \mf{u}|^2-2(\lambda \cdot \u)\Big)\,dx.
\]
We say that $\mf{u} \in H^1(\Omega,\Sigma_k)$ is a \emph{minimizer} of $J_{\lambda,\Omega}$ (with fixed trace) if 
\[
J_{\lambda,\Omega}(\mf{u}) \le J_{\lambda,\Omega}(\mf{v}) \qquad \text{for every }  \mf{v} \in H^1(\Omega, \Sigma_k) \text{ such that } \mf{v}-\mf{u} \in H_0^1(\Omega,\R^k).
\]
We also say that $\mf{u} \in H^1_{\loc}(\Omega,\Sigma_k)$ is a \emph{local minimizer} of $J_{\lambda, \cdot}$ in $\Omega$ if for every $\O'\subset\O$ it holds \[ J_{\lambda,\Omega'}(\mf{u}) \le J_{\lambda,\Omega'}(\mf{v}) \qquad \text{for every }  \mf{v} \in H^1(\Omega', \Sigma_k) \text{ such that } \mf{v}-\mf{u} \in H_0^1(\Omega',\R^k). \]
\end{Definition}

\begin{remark}\label{rem: min G T}
In the case $\lambda=0$, any minimizer of $J_{0,\O}$ is a harmonic map in $H^1(\O,\Sigma_k)$ (in the sense that it minimizes the Dirichlet energy under the segregation constraint), and belongs to the class $\cG(\Omega)$. Similarly, if $\lambda$ is positive, minimizers of $J_{\lambda,\Omega}$ belong to the class $\mathcal T_\lambda(\O)$. These facts can be checked by arguing as in Section \ref{sec: basics} (with simpler proofs, since we do not have to deal with the normalization condition on the $L^1$-norms).
\end{remark}

\subsection{Monotonicity formulae in the class \texorpdfstring{$\mathcal{T}_\lambda(\O)$}{}}
Let $\Omega \subset \R^N$ be a domain, $\lambda\in \R^k$ be positive and let $\u \in \mathcal{T}_\lambda(\Omega)$. For $x_0\in Z(\u)$, $r\in(0,\dist(x_0,\pa \Omega))$ and $t,\gamma>0$, we define
\be\label{def H D etc}
\begin{split}
    &H(\u,x_0,r):=\int_{\partial B_r(x_0)} |\mf{u}|^2 \, d\sigma,\\
    &D_{t}(\u,x_0,r):=\int_{B_r(x_0)} \Big( |\D \mf{u}|^2 - t (\lambda\cdot \u) \Big)\, dx,\\
    &N_t(\u,x_0,r):=\frac{rD_t(\u,x_0,r)}{H(\u,x_0,r)}, \qquad \text{ whenever }H(\u,x_0,r)\neq0,\\     &W_{\gamma,t}(\u,x_0,r):=\frac{D_t(\u,x_0,r)}{r^{N-2+2\gamma}}-\gamma\frac{H(\u,x_0,r)}{r^{N-1+2\gamma}}.
\end{split}
\ee
We will refer to $N_t$ and to $W_{\gamma,t}$ as the Almgren frequency function and the Weiss function, respectively. Note that these quantities also depend on $\lambda$, but this dependence is implicit in the choice of $\u$, and hence we will not explicitly stress it.
By standard computations, and using in particular the validity of the domain variation formula (Lemma \ref{l:inner}), it is not difficult to show that
\begin{align}\label{eq:H':D'}
\begin{split}
        H'(\u,x_0,r)=\;&\frac{N-1}{r}H(\u,x_0,r)+2\int_{\partial B_r(x_0)}\sum_{i}u_i\partial_\nu u_i\, d\sigma,\\
          D'_t(\u,x_0,r) =\;&\frac{N-2}{r}D_{t}(\u,x_0,r)+ 
     \int_{\partial B_r(x_0)} \Big(2|\pa_\nu \u|^2 + (2-t)(\lambda\cdot \u) \Big)\, d\sigma\\
     & +\frac{t(N-2)-2N}{r}
        \; \int_{ B_r(x_0)} 
        (\lambda \cdot \u)\, dx,
        \end{split}
\end{align}
where here and in what follows we denote by $\prime$ the derivative with respect to $r$. Moreover, by the divergence theorem
\be\label{D con teo div}
D_1(\u,x_0,r) =  \sum_i \int_{B_r(x_0)} \left(|\nabla u_i|^2 + u_i \Delta u_i\right)dx =\int_{\partial B_r(x_0)} \sum_i u_i\pa_\nu u_i\,d\sigma.
\ee

Unlike the functions in $\cG$, for the elements of $\mathcal{T}$ the Almgren frequency function is not monotone. This is due to the fact that the equation satisfied by each component $u_i$ in its positivity set is not linear, but rather of sublinear type (we refer to the discussion of the analogy of the current problem with the unstable two-phase obstacle problem). Nevertheless, as in \cite[Section 2]{SoaTer18}, the previous expressions allow us to derive a formula for the derivative of $W_{\gamma,t}$, and to obtain its monotonicity for suitable values of the parameters.

\begin{Lemma}
    It holds that
\begin{align}\label{eq:W'}
    \begin{split}
        r^{N-2+2\gamma}W'_{\gamma,t}(\u,x_0,r)= \; &2\int_{\partial B_r(x_0)} \sum_i\Big(\partial_\nu u_i-\frac{\gamma}{r}u_i\Big)^2\, d\sigma + (2-t) \int_{ \partial B_r(x_0)}(\lambda \cdot \u)\, d\sigma
        \\
        &+\frac{(N-2)t-2N+2\gamma(t-1)}{r}
        \int_{ B_r(x_0)} (\lambda \cdot \u)\, dx.
    \end{split}
\end{align}
\end{Lemma}

In the case $t=2$, this gives the monotonicity of the Weiss function for a range of parameters $\gamma$.

\begin{Lemma}\label{L:monotonicity:W_2}
    If $\gamma\ge2$, the function $W_{\gamma,2}(\u,x_0,r)$ is monotone non-decreasing with respect to $r$. Moreover, if $W_{\gamma,2}(\u,x_0,r)=const.$ for $r_1<r<r_2$ then $\u$ is $\gamma$-homogeneous with respect to $x_0$ in $B_{r_2}(x_0)\setminus { B_{r_1}(x_0)}$.
\end{Lemma}
\begin{proof}
    By \eqref{eq:W'} with $t=2$, we obtain that
    \[
     r^{N-2+2\gamma}W'_{\gamma,2}(\u,x_0,r) \ge \frac{2\gamma-4}{r}\int_{B_r(x_0)} (\lambda \cdot \u)\, dx \ge 0, 
    \]
    since $\u$ and $\lambda$ are non-negative and $\gamma\ge2$. Furthermore, if $W_{\gamma,2}(\u,x_0,r)=const.$ for $r_1<r<r_2$, then $\D u_i\cdot (x-x_0)-{\gamma}u_i=0$ in $B_{r_2}(x_0)\setminus B_{r_1}(x_0)$ for each $i=1,\dots,k$.  This means that $\mathbf{u}$ is $\gamma$-homogeneous.
\end{proof}

If instead we focus on the Weiss functional with $t=1$, the monotonicity is not guaranteed. Indeed, from formula \eqref{eq:W'} we see that 
\[
W_{\gamma,1}'(\u,x_0,r) = \Psi_\gamma(\u,x_0,r) - \Phi_\gamma(\u,x_0,r),
\]
where $\Psi_\gamma(\u,x_0,r)$ and $\Phi_\gamma(\u,x_0,r)$ are non-negative quantities defined by 
\be\label{def Phi_g}
\begin{split}
\Psi_\gamma(\u,x_0,r)&:= \frac{2}{r^{N-2+2\gamma}}\int_{\partial B_r(x_0)} \sum_i\Big(\partial_\nu u_i-\frac{\gamma}{r}u_i\Big)^2\, d\sigma + \frac{1}{r^{N-2+2\gamma}}\int_{\partial B_r(x_0)} (\lambda \cdot \u)\, d\sigma,\\
\Phi_\gamma(\u,x_0,r)&:= \frac{N+2}{r^{N-1+2\gamma}} \int_{ B_r(x_0)} (\lambda \cdot \u)\, dx.
\end{split} 
\ee

\subsection{Monotonicity formulae in the class \texorpdfstring{$\cG(\Omega)$}{}} Let $\Omega \subset \R^N$ be a domain, and let $\mf{u} \in \cG(\Omega)$. We can define the functions $H$, $D=D_0$, $N=N_0$, $W_{\gamma}=W_{\gamma,0}$ exactly as in \eqref{def H D etc}. We stress that for functions in $\cG(\Omega)$ we will always choose $t=0$, and hence we omit the dependence on $t$ in the functionals $D$, $N$, $W_\gamma$.

 The following results have been proved in \cite[Sections 2 and 3]{TaTe12} (actually, in \cite{TaTe12} a more general situation is  considered, in which each component $u_i$ solves a semilinear problem where it is positive; thus, in the present setting we obtain simplified statements with respect to those in \cite{TaTe12}).
 
 \begin{proposition}\label{prop: alm G}
 Let $\mf{u} \in \cG(\Omega)$ and $x_0 \in Z(\u)$. Then:
 \begin{enumerate}
     \item [($i$)] $H(\mf{u},x_0,r) \neq 0$ for every $r\in (0,\dist(x_0,\pa\O))$, the function $N(\mf{u},x_0,\cdot)$ is non-decreasing in $r$, and moreover
 \[
 \frac{d}{dr} \log\left(\frac{ H(\mf{u},x_0,r)}{r^{N-1}}\right)  = \frac{2 N(\mf{u},x_0,r)}{r}.
 \]
 Furthermore, $N(\u,x_0,r) = \gamma$ for all $r \in (r_1,r_2)$ if and only if $\u$ is $\gamma$-homogeneous with respect to $x_0$ in $B_{r_2}(x_0) \setminus {B_{r_1}(x_0)}$.
     \item [($ii$)] The map $Z(\u) \ni x_0 \mapsto N(\u,x_0,0^+)$ is upper semi-continuous.
     \item [($iii$)] If $x_0 \in Z(\u)$, then either $N(\u,x_0,0^+)=1$, or $N(\u,x_0,0^+) \ge 3/2$.
 \end{enumerate}
 \end{proposition}

Concerning point ($iii$), in \cite{TaTe12} it is only proved that $N(\u,x_0,0^+) \ge 1+\delta_N$ whenever $N(\u,x_0,0^+) \neq 1$, with $\delta_N>0$ a dimensional constant, and $\delta_2=1/2$. The fact that $\delta_N=1/2$ for every $N$ was subsequently proved in \cite{SoTe15} (see also \cite{OgVe}).

Proposition \ref{prop: alm G} represents the starting point for the blow-up analysis. Let $\u \in \cG(\Omega)$, $x_0 \in Z(\u)$, and let 
\[
\u_{r}(x):= \frac{\u(x_0+rx)}{\left(r^{1-N} H(\mf{u},x_0,r)\right)^{1/2}}, \qquad x \in \frac{\Omega-x_0}{r}.
\]
Note that $\|\u_r\|_{L^2(\partial B_1)}=1$ for every $r$.

\begin{Proposition}[See Section 3 in \cite{TaTe12}]\label{prop: blow-up G}
With the previous notation, for any sequence of radii $r \to 0^+$ we have that up to a subsequence $\u_r \to \u_0$ in $C^{0,\alpha}_{\loc}(\R^N)$ for every $\alpha \in (0,1)$ and strongly in $H^1_{\loc}(\R^N)$, where $\u_0 \in \cG_{\loc}(\R^N)$. Moreover, $\u_0$ is homogeneous with respect to $0$ of degree $N(\u,x_0,0^+)$. 
\end{Proposition}

\subsection{Monotonicity formulae and vanishing order}\label{S:mon:and:van}

Roughly speaking, Proposition \ref{prop: blow-up G} shows that, for functions in $\cG(\Omega)$, the value $N(\u,x_0,0^+)$ describes the vanishing order of $\u$ at $x_0$. This is a motivation for the following definition.

\begin{Definition} 
Let $\u \in H^1(\Omega,\Sigma_k)$ and $x_0 \in \Omega$. \begin{enumerate}
    \item[($i$)] We define the \emph{$N$-vanishing order of $\u$ at $x_0$} as 
\[
N_1(\u,x_0,0^+):= \lim_{r \to 0^+} N_1(\u,x_0,r),
\]
provided that the limit exists. 
    \item[($ii$)] We define the \emph{$L^2$-vanishing order of $\u$ at $x_0$} as
\[
\mathcal{V}(\mf{u},x_0):= \sup\left\{\gamma>0 \ \left| \ \limsup_{r \to 0^+} \frac{1}{r^{N-1+2\gamma}} \int_{\pa B_{r}(x_0)} \u^2\,d\sigma <+\infty \right. \right\}.
\]
\end{enumerate}
\end{Definition}

At this point we have three apparently different notions of vanishing order, which will serve for different purposes. In the end, we will be able to prove that for minimizers of \eqref{e:partition.functions} the three notions coincide (at least below the quadratic threshold). This is why we adopt a common terminology.

\begin{remark}
Note that $\mathcal{O}(\u,x_0) =0$ if $\u(x_0) \neq 0$. Furthermore, $\mathcal{O}(\u,x_0)$ is characterized by the property that
\be\label{char O gamma}
\limsup_{r \to 0^+} \frac{\|\u\|_{x_0,r}}{r^{\gamma}} = \begin{cases} 0 & \text{if }\gamma<\mathcal{O}(\u,x_0) \\ +\infty & \text{if } \gamma>\mathcal{O}(\u,x_0). \end{cases}
\ee
A similar characterization can be given for the $L^2$-vanishing order. As a consequence, it is clear that $\mathcal{O}(\u,x_0)\leq \mathcal{V}(\u,x_0)$.
\end{remark}

In what follows we show that, whenever the vanishing order $\cO(\u,x_0)$ of a local minimizer is strictly below $2$, then $\u$ behaves like a function in $\cG(\Omega)$ at an infinitesimal scale. In particular, one can establish a monotonicity formula for the Almgren frequency. To reach such a result we first prove some partial monotonicity of the Weiss-type functional at such points.

\begin{Lemma}\label{L:W:almost:monotone}
   Let $\u \in \mathcal{T}_\lambda(\Omega)$ for some positive $\lambda \in \R^k$, and let $x_0 \in Z(\u)$ be such that     \begin{equation}\label{eq:H1:V<2}
             \mathcal O:=\mathcal O(\u,x_0)= 2-\delta,
     \end{equation}
     for some $\delta>0$. Then, for every $\delta'\in(0,\delta)$, there exist $C_0>0$ and $r_0$ possibly depending on $x_0$ such that the map 
    \[
    r\mapsto W_{\mathcal{O},1}(\u,x_0,r) + C_0 r^{\delta-\delta'},
    \]
    is monotone non-decreasing for $r\in(0,r_0)$.

    Moreover, for every $\e>0$ such that $2\e\in(0,\delta -\delta')$, there exist $C_1>0$ and $r_1>0$ possibly depending on $x_0$ such that the map 
    \[
    r\mapsto W_{\mathcal{O}+\e,1}(\u,x_0,r) + C_1 r^{\delta-\delta'-2\e},
    \]
    is monotone non-decreasing for $r\in(0,r_1)$.

\end{Lemma}

\begin{proof}
    First, by assumption \eqref{eq:H1:V<2} and by the definition of vanishing order, it follows that for every $\delta'<\delta$ there exists $r_0>0$ such that
    \[
    \|\u\|_{x_0,r}\le  r^{\mathcal{O}-\delta'} \qquad \text{for all } r \in (0,r_0).
    \]
    Then, using the Poincar\'e inequality \eqref{Poinc norm}
    combined with \eqref{def Phi_g}, we get
    \[
    r^{N-2+2\mathcal{O}} W'_{\mathcal{O},1}(\u,x_0,r)\ge -({N+2}) r^{N-1} \frac{1}{r^N}\int_{B_r(x_0)}(\lambda\cdot \u)\,dx \ge -C r^{N-1}\|\u\|_{x_0,r} \ge - C  r^{N-1+\mathcal{O-\delta'}},
    \]
    whence $W'_{\mathcal{O},1}(\u,x_0,r)\ge -C r^{1-\mathcal O-\delta'}$, and the first part of the thesis follows. The proof of the second part relies on the same argument.
    \end{proof}

The previous lemma allows us to rigorously relate the vanishing order and the limit of the Weiss functionals.

\begin{lemma}\label{lem: char O W}
Let $\u \in \mathcal{T}_\lambda(\Omega)$, and let $x_0 \in Z(\u)$ be such that $\cO(\u,x_0) \le 2$. 
\begin{enumerate}
    \item [($i$)] If $\gamma \in (0, \mathcal{O}(\u,x_0))$ and $t>0$, then 
\[
\lim_{r \to 0^+} W_{\gamma,t}(\u,x_0,r) =0.
\]
     \item [($ii$)] If $\cO(\u,x_0)<2$, and $\gamma > \cO(\u,x_0)$, then
\[
\liminf_{{r \to 0^+}} W_{\gamma,1}(\u,x_0,r)=-\infty.
\]
\end{enumerate}
\end{lemma}

\begin{proof}
The first point follows from the fact that, by definition of $\mathcal{O}(\u,x_0)$ and by the Poincar\'e inequality,
\[
|W_{\gamma,t}(\u,x_0,r)| \le \frac{1+\gamma}{r^{2\gamma}} \|\u\|_{x_0,r}^2 +\frac{ |\lambda| t}{r^{N-2+2\gamma}} \int_{B_r(x_0)} |\u| \,dx \le \frac{1+\gamma}{r^{2\gamma}} \|\u\|_{x_0,r}^2 + C|\lambda| t r^{2-\gamma} \frac{\|\u\|_{x_0,r}}{r^{\gamma}} \to 0,
\]
as $r\to 0^+$, for every $\gamma<2$ and $t >0$ (here we used that $\cO(\u,x_0) \le 2$).

Regarding point ($ii$), we first observe that $W_{\gamma,1}(\mathbf{u},x_0,0^+) < 0$ implies $W_{\gamma',1}(\mathbf{u},x_0,0^+) = -\infty$ for every $\gamma' > \gamma$. Indeed, in such a case 
\[
\begin{split}
    W_{\gamma',1}(\u,x_0,r) & =\frac{W_{\gamma,1}(\u,x_0,r)}{r^{2(\gamma'-\gamma)}} +\frac{(\gamma-\gamma')H(\u,x_0,r)}{r^{N-1+2\gamma'}}
      \to -\infty,
    \end{split}
\]
as $r\to 0^+$. As a consequence, it is sufficient to show that
\begin{equation}\label{eq:transition}
    W_{\mathcal{O}+\eps,1}(\u,x_0,0^+)=-\infty,
\end{equation}
for every $\e>0$ sufficiently small. More precisely, we let $\cO(\u,x_0) =2-\delta$ and $\delta' \in (0,\delta)$, and we show \eqref{eq:transition} whenever $2\e \in (0,{\delta-\delta'})$. We fix such an $\eps$, and suppose by contradiction that there exists $\alpha \in (0, \e]$ such that 
\be\label{1551}
W_{\gamma,1}(\mathbf{u},x_0,0^+) \in [0,+\infty) \qquad \text{for all } \gamma \in (\mathcal O, \mathcal{O} + \alpha),
\ee
where the existence of the limit and its finiteness are consequences of Lemma \ref{L:W:almost:monotone}. Then, we take $\mathcal O < \gamma_1 <\gamma_2<\mathcal O+\alpha$, and consider a sequence $r_m\to0^+$ such that
\[
\lim_{m}\frac{\|\u\|_{x_0,r_m}}{r_m^{\gamma_1}}=\lim_{m}\frac{\|\u\|_{x_0,r_m}}{r_m^{\gamma_2}}=\lim_{m}\frac{\|\u\|_{x_0,r_m}}{r_m^{2}}=+\infty.
\]
Let
\[
\u_{r_m}(x):=\frac{\u(x_0+r_m x)}{\|\u\|_{x_0,r_m}}.
\]
By definition, for $j=1,2$
\begin{align*}
  \frac{ r_m^{2\gamma_j}}{\|\u\|_{x_0,r_m}^2}W_{\gamma_j,1}(\u,x_0,r_m)+ \frac{r_m^2}{\|\u\|_{x_0,r_m}} \int_{B_1} (\lambda \cdot \u_{r_m}) \, dx =\int_{B_1} |\D \u_{r_m}|^2  \, dx-
    \mathcal \gamma_j \int_{\partial B_1} 
    | \u_{r_m}|^2 \, d\sigma, 
\end{align*}
and $\|\u_{r_m}\|_{0,1}=1$ for every $m$. Thus, up to a subsequence (still denoted by $r_{m}$) we have that $\u_{r_m} \to \u_0$ strongly in $L^2(\partial B_1)$, and taking the limit in the previous identity, recalling the bounds in \eqref{1551}, it follows that 
\[
\lim_{m} \int_{B_1}  |\D \u_{r_m}|^2 \, dx = \gamma_j \int_{\partial B_1} | \u_{0}|^2\,d\sigma, \quad \text{for }j=1,2.
\]
Thus
\[
\gamma_1 \int_{\partial B_1} |\u_{0}|^2\,d\sigma = \gamma_2 \int_{\partial B_1} | \u_{0}|^2\,d\sigma,
\]
which in turn forces $\|\u_{r_m}\|_{L^2(\pa B_1)} \to \|\u_0\|_{L^2(\partial B_1)}=0$, and $\|\D \u_{r_m}\|_{L^2(B_1)} \to \gamma_j\|\u_0\|_{L^2(\partial B_1)}=0$ as well, in contradiction with $\|\u_{r_m}\|_{0,1}=1$. This shows that \eqref{1551} is not possible and, as already observed, completes the proof.
\end{proof}

\begin{remark}\label{rem: O W G}
We point out that the assumption $\cO(\u,x_0) \le 2$ is always satisfied if $\u$ is a minimizer of $J_{\lambda,\O}$, by Theorem \ref{prop: non-deg}.
Moreover, we observe that the lemma also holds for $\u \in \cG(\Omega)$ and $x_0 \in Z(\u)$, by replacing $W_{\gamma,t}$ and $W_{\gamma,1}$ with $W_\gamma$. Actually, in such a case we do not need any assumption on $\cO(\u,x_0)$, so that the alternative
\[
\lim_{r \to 0^+} W_\gamma(\u,x_0,r)=0 \quad \text{for $\gamma<\cO(\mf{u},x_0),\quad$ and } \quad \liminf_{r \to 0^+} W_\gamma(\u,x_0,r)=-\infty \quad \text{for $\gamma>\cO(\mf{u},x_0)$,}
\]
is always valid.
\end{remark}

In order to show the monotonicity of the Almgren frequency, we need to establish a more precise non-degeneracy condition at points where the vanishing order $\mathcal{O}(\u,x_0)$ is strictly less than $2$.

\begin{Lemma}\label{L:O<2:degenerate}
        Let $\u \in \mathcal{T}_\lambda(\Omega)$ for some positive $\lambda \in \R^k$, and let $x_0 \in Z(\u)$ be such that $ \mathcal O:=\mathcal O(\u,x_0)= 2-\delta,$ for some $\delta>0$. 
      Then, for every $\gamma>\mathcal O$ it follows that 
      \[\liminf_{r\to0^+}\frac{\|\u\|_{x_0,r}}{r^{\gamma}}=\infty.
      \]
\end{Lemma}

\begin{proof}
The thesis is non-trivial only if we assume that there exists a sequence $r_n\to0^+$ such that 
\begin{equation}\label{eq:sequence:H^1:VO}
   \frac{\|\u\|_{x_0,r_n}}{r_n^{\mathcal O}} \to 0, \qquad \text{as } n\to\infty.
\end{equation}
Hence, under \eqref{eq:sequence:H^1:VO}, let
\[
\u_r(x):=\frac{\u(x_0+r x)}{\|\u\|_{x_0,r}}.
\]
Since $\|\mf{u}_r\|_{0,1}=1$ for every $r$, and
\begin{align}\label{eq:Weiss:rescaled}
    \begin{split}
    W_{\mathcal O,1}(\u,x_0,r_n)&= \frac{\|\u\|_{x_0,r_n}^2}{r_n^{2 \mathcal O}} \Big(\int_{B_1} |\D \u_{r_n}|^2 \, dx
    -
    \mathcal O \int_{\partial B_1} | \u_{r_n}|^2 \, d\sigma
    \Big)- r_n^{2-\mathcal O} \frac{\|\u\|_{x_0,r_n}}{r_n^{ \mathcal O}} \int_{B_1} (\lambda \cdot \u_{r_n}) \, dx,
    \end{split}
    \end{align}
    condition \eqref{eq:sequence:H^1:VO} implies that $W_{\mathcal{O},1}(\u,x_0,0^+)=0$, and by the almost monotonicity of the Weiss functional proved in Lemma \ref{L:W:almost:monotone} it follows that
\[
       \frac{D_1(\u,x_0,r)}{r^{N-2+2\mathcal O}}\ge W_{\mathcal{O},1} (\u,x_0,r) \ge -Cr^{\delta - \delta'},
\]
   where $\delta'\in(0,\delta)$. We also recall that $W_{\gamma,1}(\u,x_0,0^+)=-\infty$ for every $\gamma >\cO$, by Lemma \ref{lem: char O W}. 

Let $\gamma=\mathcal O+\e$, where $2\e \in (0,{\delta-\delta'})$. As a consequence of what we have proved so far, we have
\begin{align*}
    -\infty&= W_{\gamma,1}(\u,x_0,0^+) = \lim_{r\to0^+} \Big(
    \frac{D_1(\u,x_0,r)}{r^{N-2+2\gamma}}
- \gamma
     \frac{H(\u,x_0,r)}{r^{N-1+2\gamma}}   
    \Big)\\
    &\ge \limsup_{r\to0^+} 
    \Big( -C
    r^{\delta-\delta'-2\e}
- \gamma
     \frac{H(\u,x_0,r)}{r^{N-1+2\gamma}}   
    \Big) = -\gamma \liminf_{r\to 0^+}   \frac{H(\u,x_0,r)}{r^{N-1+2\gamma}}.
\end{align*}
Since $\|\u\|_{x_0,r}^2\ge r^{1-N}H(\u,x_0,r)$, the thesis follows for any $\gamma$ as before. This directly implies the thesis also for any larger exponent.
\end{proof}

\begin{Lemma}\label{p:Almgren-below}
Let $\u \in \mathcal{T}_\lambda(\Omega)$ for some $\lambda \in \R^k$ positive, and let $x_0 \in Z(\u)$ be such that 
\[
\mathcal{O}(\u,x_0)= 2-\delta,
\]
for some $\delta >0$. Then, for $\delta'\in(0,\delta)$, there exist $C>0$ and $r_0>0$ possibly depending on $x_0$ such that $N_1(\u,x_0,r) +1 \ge 0$ and 
\[
r\mapsto e^{C r^{\delta'} }(N_1(\u,x_0,r) + 1),
\]
is monotone non-decreasing for $r\in (0,r_0)$. Therefore, the limit 
\[
N_1(\u,x_0,0^+):= \lim_{r \to 0^+} e^{C r^{\delta'} }(N_1(\u,x_0,r) + 1) -1.
\]
exists and is finite.
\end{Lemma}

\begin{proof}
In this proof, to simplify the notation we omit the dependence of $H$, $D_1$ and $N_1$ from $\u$ and $x_0$, which are fixed. Firstly, by \eqref{eq:H':D'} and \eqref{D con teo div}, we obtain that
\begin{align*}
    \frac{d}{dr} N_1(r) =& \frac{1}{(H(r))^2} \left( (N-2)H(r) D_1(r) + D_1(r) H(r) + 2r H(r) \int_{\partial B_r(x_0)} |\pa_\nu \u|^2\,d\sigma \right.\\
    & \qquad \qquad \qquad \left. -(N-1) D_1(r) H(r) -2r D_1(r) \int_{\partial B_r(x_0)} \sum_i u_i \pa_\nu u_i\,d\sigma\right) \\
    &  + \frac{1}{H(r)} \left( r \int_{\partial B_r(x_0)}(\lambda\cdot \u)\, d\sigma -(N+2) \int_{ B_r(x_0)} 
        (\lambda\cdot \u)\, dx\right)\ge  -\frac{N+2}{H(r)} \int_{ B_r(x_0)} 
        (\lambda\cdot \u)\, dx,
\end{align*}
where we used the Cauchy-Schwarz inequality. Now, by the Poincar\'e inequality \eqref{Poinc norm}, we observe that
\[
\begin{split}
\|\u\|_{x_0,r}^2 &= \frac{1}{r^{N-2}}D_1(r) + \frac{1}{r^{N-1}} H(r) + \frac{1}{r^{N-2}} \int_{B_r(x_0)} (\lambda\cdot \u)\, dx\\
        & \le \frac{1}{r^{N-2}}D_1(r) + \frac{1}{r^{N-1}} H(r) + C r^2 |\lambda|\|\u\|_{x_0,r},
 \end{split}
 \]
whence 
\[
r^2 \|\u\|_{x_0,r} \left( \frac{\|\u\|_{x_0,r}}{r^2} - C |\lambda| \right) \le \frac{1}{r^{N-2}}D_1(r) + \frac{1}{r^{N-1}} H(r).
\]
Since $\cO(\u,x_0)<2$, by Lemma \ref{L:O<2:degenerate} the term in the brackets on the left hand side is positive whenever $r<r_0$, with $r_0$ sufficiently small (which could depend on $x_0$). This proves that for any such $r$ also the right hand side is positive; moreover, we also obtain the estimate
\be\label{2731}
\frac{C}{|\lambda|}\left(\frac{1}{r^N} \int_{B_r(x_0)} (\lambda \cdot \u) \,dx\right) \left( \|\u\|_{x_0,r} - C |\lambda| r^2 \right) \le \frac{1}{r^{N-2}}D_1(r) + \frac{1}{r^{N-1}} H(r),
\ee
for all $r \in (0,r_0)$, and hence we can compute
\[
\begin{split}
 \frac{d}{dr} \log(N_1(r)+1) &\geq -\ddfrac{{r^{1-N}}H(r)}{r^{2-N}D_1(r) + r^{1-N} H(r)} \cdot \ddfrac{(N+2)}{r^{1-N} H(r)} \cdot \frac{1}{r^{N-1}} \int_{B_r(x_0)}(\lambda \cdot \u)\,dx\\
& \ge - \frac{C|\lambda| \, r}{\left( \|\u\|_{x_0,r} - C |\lambda| r^2 \right)} = -\frac{C |\lambda|}{r}\left(\frac{\|\u\|_{x_0,r}}{r^2}-C |\lambda|\right)^{-1},
\end{split}
 \]
for any $r \in (0,r_0)$. If necessary replacing $r_0$ with a smaller quantity, by using Lemma \ref{L:O<2:degenerate} we can suppose that $\|\u\|_{x_0,r} \ge C_1 r^{2-\delta'}$, for $\delta'\in(0,\delta)$, for all $r \in (0,r_0)$, for some $C_1>0$; namely
\be\label{e:51}
\frac{\|\u\|_{x_0,r}}{r^2}-C |\lambda| \ge C_1 r^{-\delta'}-C \ge \frac{C_1}{2} r^{-\delta'}  \qquad \text{for all } r \in (0,r_0),
\ee
whence it follows that 
\[
 \frac{d}{dr} \log(N_1(r)+1) \ge - \frac{2C |\lambda|}{C_1 r^{1-\delta'}} \qquad \text{for all } r \in (0,r_0).
 \]
By integrating, we finally obtain the desired almost monotonicity property for the frequency function. 
\end{proof}

In light of the previous result, we can further say that the three notions of vanishing orders we introduced coincide below the quadratic threshold.

\begin{Lemma}
Let $\u \in \mathcal{T}_\lambda(\Omega)$, and let $x_0 \in Z(\u)$ be such that at least one of the conditions $\mathcal{O}(\u,x_0)<2$ or $\cV(\u,x_0)<2$ holds; then
\be\label{ordini uguali}
    \mathcal{O}(\u,x_0)=\mathcal{V}(\u,x_0) =N_1(\u,x_0,0^+).
\ee
\end{Lemma}
\begin{proof}
We first show that $\cO(\u,x_0) = \cV(\u,x_0)$. We have already observed that $\mathcal{O}(\u,x_0)\leq \mathcal{V}(\u,x_0)$ (see the discussion after \eqref{char O gamma}). Thus, suppose by contradiction that  $\mathcal{O}(\u,x_0)< \mathcal{V}(\u,x_0)$ and consider $k =2-\alpha \in (\mathcal{O}(\u, x_0), \mathcal{V}(\u, x_0))$, with $\alpha >0$. As in the proof of Lemma \ref{p:Almgren-below}, we have that
    \[
    \|\u\|_{x_0,r} \left( \|\u\|_{x_0,r} - C |\lambda| r^2\right) \le \frac{1}{r^{N-2}}D_1(r) + \frac{1}{r^{N-1}} H(r) = \frac{1}{r^{N-1}} H(r) (N_1(r)+1),
    \]
    and the term inside the brackets on the left hand side is positive for $r \in (0,r_0)$. We deduce that
    \[
    \frac{\norm{\u}{x_0,r}}{r^{k}}\leq \left(\frac{N_1(r)+1}{\norm{\u}{x_0,r}-C |\lambda| r^2}\cdot r^{k}\right)    \frac{H(r)}{r^{N-1+2k}},
    \]
    As in \eqref{e:51}, if necessary, replacing $r_0$ with a smaller quantity, we can suppose that
    \[
    \norm{\u}{x_0,r} - C|\lambda| r^2 \geq \frac{C_1}{2} r^{k}  \quad \text{for all } r \in (0,r_0).
    \]
    for every $r \in (0,r_0)$, for some $C_1>0$. Therefore, using also the almost-monotonicity of the Almgren function, we obtain that
    \[
    \frac{\norm{\u}{x_0,r}}{r^{k}}\leq C (N_1(r)+1 )\frac{H(r)}{r^{N-1+2k}}\leq C e^{C r_0^\alpha}(N_1(R_0)+1)\frac{H(r)}{r^{N-1+2k}}.
    \]
    On the one hand, since $k>\cO(\u,x_0)$, by Lemma \ref{L:O<2:degenerate} the left hand side tends to $+\infty$ as $r \to 0^+$; but on the other hand, as $k<\cV(\u,x_0)$, the limit of the right hand side is $0$. Thus, we reached the desired contradiction.

Regarding the equality $\cV(\u,x_0) = N_1(\u,x_0,0^+)$, the proof exploits the almost monotonicity of the Almgren's frequency in a standard way. We omit the details, since the argument coincides with that of \cite[Corollary 5.3]{Tor} in the case $s=q=1$. 
\end{proof}

\begin{remark}
Thanks to Proposition \ref{prop: alm G}, the previous arguments can be adapted to show that equality \eqref{ordini uguali} also holds for functions in $\cG(\Omega)$ at any nodal point $x_0$ (without any assumption on the vanishing order).
\end{remark}

Finally, as a consequence of the Almgren frequency formula, we state a doubling property which will be useful in the blow-up analysis.

\begin{lemma}\label{lem: doub}
    Let $\u \in \mathcal{T}_\lambda(\Omega)$ for some $\lambda \in \R^k$ positive, and let $x_0 \in Z(\u)$ be such that $\cO(\u,x_0) <2$. There exists $C>0$ and $r_0>0$ possibly depending on $x_0$ such that
    \[
    \frac{1}{r_2^{N-1}}H(\u,x_0,r_2) \le \frac{1}{r_1^{N-1}}H(\u,x_0,r_1) \left(\frac{r_2}{r_1}\right)^{2C},
    \]
    for every $0<r_1<r_2<r_0$.
\end{lemma}
\begin{proof}
In view of Lemma \ref{p:Almgren-below}, the proof is analogous to that of \cite[Corollary 2.6]{TaTe12}.   
\end{proof}

\subsection{Uniform monotonicity formulae when \texorpdfstring{$\mathcal{O}(\mathbf{u},x_0)=1$}{}}

In Section \ref{S:mon:and:van}, assuming that the vanishing order $\mathcal{O}(\u,x_0)<2$, we established the (almost)-monotonicity of the Weiss and Almgren functionals on an interval $(0,r_0)$. However, the radius $r_0>0$ involved in these results depends on the point $x_0$. In this section, we prove that, if we restrict our analysis to points where $\mathcal{O}(\u,x_0)=1$, this $r_0>0$ can be chosen independently of $x_0$. These results allow us to establish the compactness of blow-up sequences with variable centers under the condition that the vanishing order is equal to $1$, a  crucial result for the regularity of the zero set $Z(\u)$.

\begin{lemma}\label{L:W:almost:monotone:O=1}
   Let $\u \in \mathcal{T}_\lambda(\Omega)$ for some positive $\lambda \in \R^k$, and let $x_0 \in Z(\u)$ be such that $\mathcal{O}:=\mathcal{O}(\u,x_0)=1$. Then, for every $\e\in[0,1/2)$ there exists $C>0$ independent of $x_0$ such that the map 
    \[
    r\mapsto W_{1+\e,1}(\u,x_0,r) + C r^{1-2\e},
    \]
    is monotone non-decreasing for $r\in(0,\dist(x_0,\partial \O))$.   
\end{lemma}

\begin{proof}
 The Lipschitz regularity of $\u$ ensures that there exists a constant $C>0$ such that $\|\u\|_{x_0,r} \le C r$ for every $r \in (0,\dist(x_0,\partial \O))$.
    Therefore, the thesis follows by arguing exactly as in Lemma \ref{L:W:almost:monotone}.
\end{proof}

In the following lemma, given a compact set $K$, we denote by $K_\delta$ the $\delta$ neighbourhood of $K$. 

\begin{lemma}
    Let $\u \in \mathcal{T}_\lambda(\Omega)$ for some positive $\lambda \in \R^k$. Let $K\subset \O$ be a compact set such that $\mathcal{O}(\u,x)=1$ for every $x\in K_\delta\cap Z(\u)$. Then, for every $\e>0$, there exist $r_0>0$ and $C>0$ (independent of $x$) such that
    \be\label{eq:O=1:nondegeneracy}
\frac {\|\u\|_{x,r}}{r^{1+\e}}\ge C,\qquad \text{for every $r\in(0,r_0)$ and $x\in  K\cap Z(\u)$}.
    \ee
\end{lemma}

\begin{proof}

Arguing by contradiction, let us assume that there exist $\e\in(0,1/2)$, a sequence of points $\{x_n\}\subset K\cap Z(\u)$ converging to $x_0\in K \cap Z(\u)$ with $\mathcal{O}(\u,x_n)=\mathcal{O}(\u,x_0)=1$, and a sequence of radii $r_n\to 0^+$ such that
\[
    \frac{\|\u\|_{x_n,r_n}}{r_n^{1+\e}} \to 0 \quad \text{as } n\to\infty.
\]
Following the same computation as in \eqref{eq:Weiss:rescaled}, we deduce that
\be\label{3051}
    W_{1+\e,1}(\u,x_n,r_n)\to 0 \quad \text{as } n\to\infty.
\ee
To reach a contradiction, we introduce the sequence of continuous functions on $K \cap Z(\u)$ defined by
\[
G_n(x):=\max\left\{W_{1+\e,1}(\u,x,r_n)+Cr_n^{1-2\e},-1\right\},
\]
where $C>0$ is the constant provided by Lemma \ref{L:W:almost:monotone:O=1}; in particular, $C$ is independent of both $x$ and $n$. By the same lemma, for every $x\in K\cap Z(\u)$, the map
\[
r\mapsto W_{1+\e,1}(\u,x,r)+Cr^{1-2\e},
\]
is non-decreasing in $(0,r_0)$, with $r_0>0$ independent of $x$. Therefore, up to extracting a decreasing subsequence, still denoted by $r_n$, the sequence $\{G_n\}$ is pointwise non-increasing on $K\cap Z(\u)$; namely,
\[
G_{n+1}(x)\leq G_n(x)
\qquad\text{for every }x\in K\cap Z(\u).
\]
Moreover, for each fixed $x\in K\cap Z(\u)$, Lemma \ref{lem: char O W} together with the almost-monotonicity of the Weiss functional yields
\[
W_{1+\e,1}(\u,x,r)+Cr^{1-2\e}\to -\infty
\qquad\text{as } r\to0^+.
\]
Consequently, $G_n(x)\to -1$ for every $x\in K\cap Z(\u)$. Since the functions $G_n$ are continuous and converge monotonically to the continuous function $-1$, Dini's convergence theorem implies that $G_n\to -1$ uniformly on $K\cap Z(\u)$.
This contradicts \eqref{3051}, according to which $G_n(x_n)\to0$ as $n\to\infty$.
\end{proof}

As an immediate consequence of the preceding uniform estimates, we can refine the Almgren frequency formula at points $x$ with vanishing order $\mathcal O(\u,x)=1$.

\begin{lemma}\label{lem: alm unif}
Let $\u \in \mathcal{T}_\lambda(\Omega)$ for some positive $\lambda \in \R^k$. Let $K\subset \O$ be a compact set such that $\mathcal{O}(\u,x)=1$ for every $x\in K_\delta\cap Z(\u)$. Then for every $\e\in(0,1/2)$ there exist $C>0$ and $r_0>0$ (independent of $x$) such that the map
\[
r\mapsto e^{C r^{1-\e }}(N_1(\u,x,r) + 1),
\]
is monotone non-decreasing for $r\in (0,r_0)$, for every $x \in K \cap Z(\u)$.
\end{lemma}

\begin{proof}
    The proof is exactly as that of Lemma \ref{p:Almgren-below}, using the estimate \eqref{eq:O=1:nondegeneracy} instead of  Lemma \ref{L:O<2:degenerate}.
\end{proof}

This allows us to establish a uniform-in-$x$ doubling property:

\begin{lemma}\label{lem: doub unif}
Let $\u \in \mathcal{T}_\lambda(\Omega)$ for some positive $\lambda \in \R^k$. Let $K\subset \O$ be a compact set such that $\mathcal{O}(\u,x)=1$ for every $x\in K_\delta\cap Z(\u)$. 
    There exist $C>0$ and $r_0>0$ (independent of $x$) such that
    \[
    \frac{1}{r_2^{N-1}}H(\u,x,r_2) \le \frac{1}{r_1^{N-1}}H(\u,x,r_1) \left(\frac{r_2}{r_1}\right)^{2C},
    \]
    for every $0<r_1<r_2<r_0$, for every $x \in K \cap Z(\u)$.
\end{lemma}

\section{Upper semi-continuity of the vanishing order map}\label{sec: upper}

A key ingredient in both the blow-up analysis and the study of the singular set is the following upper semi-continuity property for the vanishing order map $x_0 \in Z(\u) \mapsto \cO(\u,x_0)$. It is convenient to recall the Definitions \ref{def T}, \ref{def G} and \ref{def J} of $\mathcal{T}_{\lambda}(\O) $, $\cG(\Omega)$ and $J_{\lambda,\O}$, respectively.

\begin{proposition}\label{prop: upper}
Let $\{\lambda_n\} \subset \R^k$ be non-negative, and for each $n$ let $\u_n \in \mathcal{T}_{\lambda_n}(B_3)$. Let $x_n \in Z(\u_n) \cap B_1$, $x_n \to \xi$, $\lambda_n \to \lambda_0$, and suppose that $\u_n \to \u_0$ in $H^1_{\loc}(B_3)$ and in $C^{0,\alpha}_{\loc}(B_3)$ for some $\alpha \in (0,1)$, as $n\to\infty$, where $\u_0$ is a minimizer of $J_{\lambda_0,B_2}$. Suppose in addition that either $\lambda_0$ is positive, or $\lambda_0=0$, and let
\[
\bar d:= \limsup_n \mathcal{O}(\u_n,x_n).
 \]
 Then:
\begin{enumerate}
\item[($i$)] if $\bar d \ge 2$, then $\mathcal{O}(\u_0,\xi) \ge 2$;
\item[($ii$)] if $\bar d<2$, then $\mathcal{O}(\u_0,\xi) \ge \bar d$.
\end{enumerate}
\end{proposition}

The proof of this proposition is rather long and will take the rest of the section. It is well known that statements like Proposition \ref{prop: upper} are extremely useful for the study of the nodal set of solutions to elliptic equations, see \cite{Han94, Han00}. For instance, the simplest consequence of the proposition is that we can conveniently divide $Z(\u)$ into regular and singular part according to the vanishing order, and that the singular set is relatively closed (see Theorem \ref{thm: fb}).

For sequences of functions in $\cG(\Omega)$, the result follows rather directly from the Almgren frequency formula, but this is not valid in the present setting when we deal with nodal points at the quadratic threshold (or above). The same issue occurs when dealing with sublinear equations, see \cite{SoaTer18} and \cite{SoaTor24}. In what follows, we adapt an ad-hoc iterative argument based on the study of the Weiss-type functionals $W_{\gamma,1}$, introduced in \cite{SoaTer18} for sublinear scalar equations. We recall the expression of $W_{\gamma,1}'$, see \eqref{eq:W'}, and in particular the definition of $\Phi_\gamma$, equation \eqref{def Phi_g}. We start with a few auxiliary statements, whose proofs are analogous to those of \cite[Lemma 5.3, 5.4 and 5.5]{SoaTer18}, and hence are omitted.

\begin{lemma}\label{lem 2}
Let $\u \in \mathcal{T}_\lambda(B_1(x_0))$ for some $\lambda>0$. Suppose that there exist $ \sigma \in (0,2)$ and $\bar C>0$ such that 
\[
H(\u,x_0,r) \le \bar C r^{N-1+2\sigma} \quad \text{for every $r \in (0,1)$}.
\]
Then there exists a constant $C=C(N)>0$ such that 
\[
\int_0^r \Phi_\gamma(\u,x_0,s)\,ds \le \frac{C |\lambda| {\bar{C}}^{1/2} }{2+\sigma -2 \gamma} r^{ 2+\sigma -2 \gamma},
\]
for every $\gamma \in \left[\sigma, \frac{2+\sigma}{2} \right)$ and every $r \in (0,1)$.
\end{lemma}

\begin{lemma}\label{lem 3}
Let $\u \in \mathcal{T}_\lambda(B_1(x_0))$ for some $\lambda\in \R^k$ positive. 
Suppose that 
\begin{enumerate}
\item[($i$)] $W_{\gamma,1}(\u,x_0,0^+) = 0$; 
\item[($ii$)] $\int_0^1 \Phi_\gamma(\u,x_0,s)\,ds \le \bar C$.
\end{enumerate}
for some $\gamma \in (0,2)$ and $\bar C>0$. Then there exist $\tilde C>0$ depending on $N$, $\bar C$, and upper bounds on $\|\u\|_{x_0,1}$ and on $|\lambda|$, such that 
\[
\frac{H(\u,x_0,r)}{r^{N-1+2(\gamma-\eps)}} \le \tilde C r^{2\eps} |\log r|,
\]
for every $\eps>0$ and $r \in (0,1)$.
\end{lemma}

\begin{lemma}\label{lem: seq exp}
Let 
\begin{equation}
\label{seq exp}
\begin{cases}
\sigma_0=\alpha \in (0,1) \\
\displaystyle\sigma_n=\frac12\left(\frac{2+ \sigma_{n-1}}{2}\right) + \frac{\sigma_{n-1}}2 & \text{if }n \ge 1.
\end{cases}
\end{equation}
The sequence $\{\sigma_n\}$ is monotone increasing and converges to $2$ as $n\to\infty$. Moreover, $\sigma_n < (2+ \sigma_{n-1})/2$ for every $n$.
\end{lemma}

We are ready to proceed with the proof:

\begin{proof}[Proof of Proposition \ref{prop: upper}]
For concreteness, we focus on point ($i$); point ($ii$) can be proved with minor changes. We denote by 
\[
d_n:= \mathcal{O}(\u_n,x_n) \quad \text{and} \quad d := \mathcal{O}(\u_0,\xi).
\]
Suppose by contradiction that $d <2$. Up to a subsequence, we can suppose that $d_n \to \bar d \ge 2$. By assumption, $\{\u_n\}$ is a bounded sequence in $C^{0,\alpha}(\overline{B_{3/2}})$. Then, since $\u_n(x_n)=0$ for every $n$, we have that
\[
H(\u_n,x_n,r) \le C \int_{\partial B_r(x_n)} |x-x_n|^{2\alpha} \, d \sigma \le C  r^{N-1+2\alpha} =: C_0 r^{N-1+2\alpha},
\]
for every $r \in (0,1)$ and $n \in \N$. We stress that $C_0$ is independent of $n$. 

At this point we recall the definition of $\sigma_0$ and $\sigma_1$ from \eqref{seq exp}, and apply Lemma \ref{lem 2} with $\sigma=\sigma_0=\alpha$ and $\gamma= \sigma_1 + \delta_1$, where $\delta_1>0$ is chosen in such a way that $\sigma_1+\delta_1 < (2+\sigma_0)/2$: we infer that
\begin{equation}\label{hp 2 lem 3 1}
\int_0^r \Phi_{\sigma_1+\delta_1}(\u_n,x_n,s)\,ds \le \frac{\bar C_0}{2+\sigma_0 -2(\sigma_1+\delta_1)} r^{2+\sigma_0 -2(\sigma_1+\delta_1)},
\end{equation}
for a positive constant $\bar C_0$ independent of $n$. Note that $(2+\sigma_0)/2 <2$, and hence at least for $n$ sufficiently large
\begin{equation}\label{hp 1 lem 3 1}
W_{\sigma_1+\delta_1}(\u_n,x_n,0^+) =0,
\end{equation}
by Lemma \ref{lem: char O W}. Indeed, since $d_n \to \bar d \ge 2$, we have that $d_n > \sigma_1+\delta_1$ for sufficiently large $n$. Equations \eqref{hp 2 lem 3 1} and \eqref{hp 1 lem 3 1} enable us to apply Lemma \ref{lem 3} with $\gamma=\sigma_1+\delta_1$ and $\eps=\delta_1$, deducing that there exists $C_1>0$ depending on $N$ and on $\sup_n\{\|\u_n\|_{x_n,1}\}$ (in particular, $C_1$ is independent of $n$) such that
\[
H(\u_n,x_n,r) \le C_1 r^{N-1+2\sigma_1} r^{2 \delta_1 } |\log r| \le C_1 r^{N-1+2\sigma_1}.
\]
The previous argument can be iterated as follows: first, we observe that either $(2+\sigma_1 )/2 > d$, or not. 

\emph{Case 1) $(2+\sigma_1 )/2 > d$.} We apply Lemma \ref{lem 2} with $\sigma=\sigma_1$ and any $\gamma \in (\sigma_1,d+\eps)$ with $\eps>0$ so small that $d+\eps< (2+\sigma_1 )/2$: we deduce that
\begin{equation}\label{concl}
\int_0^r \Phi_{\gamma}(\u_n,x_n,s)\,ds \le \frac{\bar C_1}{2+\sigma_1  -2\gamma} r^{2+\sigma_1  -2\gamma},
\end{equation}
for a positive constant $\bar C_1$ independent of $n$. Now, recalling the definition of $\Phi_\gamma$ (see \eqref{def Phi_g}), we have $W_{\gamma,1}'(\u_n,x_n,r) \ge -\Phi_\gamma(\u_n,x_n,r)$, whence
\begin{align*}
W_{\gamma,1}(\u_n,x_n,r) & \ge W_{\gamma,1}(\u_n,x_n,0^+) -  \int_{0}^r  \Phi_{\gamma}(\u_n,x_n,s)\,ds \ge - C r^{2+\sigma_1  -2\gamma},
\end{align*}
for large $n$, where the last inequality follows by \eqref{concl} and Lemma \ref{lem: char O W} (since $d_n \to \bar d \ge 2$, we have that $d_n >(2+\sigma_1 )/2>d+\eps>\gamma$ for sufficiently large $n$). Here $C$ denotes a positive constant $C$ independent of $n$, and the inequality holds for any $r \in (0,1)$. Now we take the limit as $n\to\infty$: by $H^{1}$ and uniform convergence, we deduce that $W_{\gamma,1}(\u_0,\xi,r) \ge  - C r^{2+\sigma_1  -2\gamma}$, and taking further the limit as $r \to 0^+$, we finally obtain 
\[
\liminf_{r \to 0^+} W_{\gamma,1}(\u_0,\xi,r) \ge 0,
\]
for every $\gamma \in (\sigma_1,d+\eps)$, with $d+\eps <2$. By Lemma \ref{lem: char O W} (see also Remark \ref{rem: O W G}), it follows that $\mathcal{O}(\u_0,\xi)$ is larger than any such $\gamma$, and in particular $\mathcal{O}(\u_0,\xi) \ge d+\eps/2$, in contradiction with the fact that $\mathcal{O}(\u_0,\xi)=d$.

\emph{Case 2) $(2+\sigma_1)/2 \le d$.} In this case we apply Lemma \ref{lem 2} with $\sigma=\sigma_1$ and $\gamma = \sigma_2+\delta_2$, $\sigma_2$ defined by \eqref{seq exp}, and $\delta_2>0$ small so that $\sigma_2+\delta_2 < (2+\sigma_1)/2$. We deduce that
\begin{equation}\label{hp 2 lem 3 2}
\int_0^r \Phi_{\sigma_2+\delta_2}(\u_n,x_n,s)\,ds \le \frac{\bar C_1}{2+\sigma_1 -2(\sigma_2+\delta_2)} r^{2+\sigma_1  -2(\sigma_2+\delta_2)},
\end{equation}
for $\bar C_1>0$ independent of $n$. Moreover, being $\sigma_2+\delta_2 < d<d_n$ eventually, we have that for every $n$ large
\begin{equation}\label{hp 1 lem 3 2}
W_{\sigma_2+\delta_2,1}(\u_n,x_n,0^+) = 0,
\end{equation}
by Lemma \ref{lem: char O W}. Equations \eqref{hp 2 lem 3 2} and \eqref{hp 1 lem 3 2} enable us to apply Lemma \ref{lem 3} with $\gamma=\sigma_2+\delta_2$ and $\eps=\delta_2$, deducing that there exists $C_2>0$ such that
\[
H(\u_n,x_n,r) \le C_2 r^{N-1+2\sigma_2} r^{2\delta_2} |\log r| \le C_2 r^{N-1+2\sigma_2}.
\]
At this point we check whether $(2+\sigma_2)/2 >d$ or not. If yes, we follow Case 1 to reach a contradiction. If not, we iterate the previous argument once again obtaining
\[
H(\u_n,x_n,r) \le C_3 r^{N-1+2\sigma_3},
\]
and so on. By Lemma \ref{lem: seq exp}, we are sure that there exists $j \in \N$ such that $(2+\sigma_j)/2> d$, so that the proof is complete after a finite number of iterations.
\end{proof}

\section{Blow-up analysis}\label{sec: blow-up}

We begin with the following lemma which implies the compactness of blow-up sequences under $H^1$-bounds.

\begin{lemma}\label{L:compactness}
Let $\u$ be a minimizer of \eqref{e:partition.functions}, 
let $\{x_n\}\subset Z(\u)$, and suppose that $r_n,\rho_n\to0^+$ as $n\to\infty$. 
Suppose moreover that 
\[
\u_n(x):=\frac{1}{\rho_n}\u(x_n+r_n x),
\]
is well defined for $x \in B_{2R}$, for every $n$. If 
\begin{equation}\label{eq:ass:compactness}
    \liminf_{n}\frac{\rho_n}{r_n^2}\geq c ,\qquad \text{and} \qquad \|\u_{n}\|_{0,2R}\le C, 
\end{equation}
for some constants $c,C>0$ independent of $r$, then up to a subsequence
\[
\u_{n}\to \u_0,\quad \text{strongly in }  H^1(B_R) \text{ and in }C^{0,\alpha}(\overline{B_R}).
\]
\end{lemma}

\begin{proof}Without loss of generality, we prove the statement for $R=1$, to simplify the notation. 
     Let us fix $i\in\{1,\dots,k\}$.
    By Corollary \ref{cor:EL eq}, we have that $u_{i,n}$ satisfies
\begin{equation}\label{eq:compact:1}
            -\Delta u_{i,n}= \lambda_{2,1}(\omega_{i})\frac{r_n^2}{\rho_n}\ind_{\{u_{i,n}>0\}} - \mu_{i,n} \quad\text{in }B_2,
    \end{equation}
    where $\mu_{i,n}$ is a non-negative Radon measure supported in $\{u_{i,n}=0\} \cap B_2$. In other words, for every $\varphi\in C_c^\infty(B_2)$, we have
    \[
    \int_{B_2}\varphi\, d\mu_{i,n} =  \int_{B_2} \left(\D u_{i,n}\cdot \D\varphi-\lambda_{2,1}(\omega_{i}) \frac{r_n^2}{\rho_n}\varphi\ind_{{\{u_{i,n}>0\}}}\right)\,dx.
    \]
    Hence, using the assumption \eqref{eq:ass:compactness}, it follows that $\mu_{i,n}$ is locally finite, uniformly with respect to $n$.

    From now on, we will fix an index $i$ and omit the dependence on $i$ to simplify the notation. Moreover, the following properties hold up to a subsequence; we will not mention this fact every time. 
    By \eqref{eq:ass:compactness}, we have that $u_{n}\rightharpoonup u_{0}$ weakly in $H^{1}(B_2)$ and a.e. in $B_2$; thus, there exists $\bar x \in B_{3/2}$ such that $\{u_n(\bar x)\}$ is bounded. Furthermore, by the Lipschitz estimate \eqref{eq:lip:estimates} it follows that $\| \nabla u_{n} \|_{L^\infty(B_{3/2})}$ is bounded as well, so that $u_n \to u_0$ in ${C^{0,\alpha}}(\overline{B_{3/2}})$ for every $\alpha\in(0,1)$. Now, let $\eta\in C_c^\infty(B_{3/2})$, and let us test equation \eqref{eq:compact:1} with $\eta (u_{n}-u_{0})$. Then
\[
    \int_{B_{3/2}} \D u_{n} \cdot \nabla(\eta (u_n-u_{0}))\, dx = \int_{B_{3/2}}\lambda_{2,1}(\omega_{i}) \frac{r_n^2}{\rho_{n}}\eta (u_{n}-u_{0})\ind_{\{u_{n}>0\}}\,dx - \int_{B_{3/2}}\eta (u_{n}-u_{0})\,d\mu_{n} \to 0,
\]
where we used \eqref{eq:ass:compactness}, the ${C^{0,\alpha}}(\overline{B_{3/2}})$ convergence, and the fact that $\mu_{n}$ is locally finite, uniformly in $n$.
Therefore, using again the ${C^{0,\alpha}}$ and weak $H^1$ convergences, we finally obtain
\[
\int_{B_{3/2}}  \eta |\D u_{n}|^2\,dx = \int_{B_{3/2}} \eta \D u_{n}\cdot \D u_{0}\, dx + \int_{B_{3/2}} (u_{0}-u_{n})\D u_{n}\cdot \D \eta \,dx +o(1) \to \int_{B_{3/2}}  \eta |\D u_{0}|^2\,dx.
\]
Since $\eta$ was arbitrarily chosen, this implies that $u_{n}\to u_{0}$ in $H^1(B_1)$. 
\end{proof}

The following lemma is a useful result that allows for the construction of segregated competitors in minimization problems. We refer to \cite[Theorem 1.2]{OgVe} for a similar construction.

\begin{Lemma}\label{L:construction:commpetitor}
    Assume that $\{\u_n\}\subset H^{1}(B_R,\Sigma_k) $ is a  non-negative sequence such that $ \u_n\to \u$ strongly in $H^1(B_R)$, and that there exists a constant $C>0$ such that $[\u_n]_{C^{0,1}(\overline{B_R})}\le C$ for all $n$.
    Assume that $\mathbf{v} \in H^{1}(B_R,\Sigma_k)$ and $\mathbf v=\u$ on $\partial B_R$.
    
    Then, there exists a sequence $\{\mathbf{v}_n\}\subset H^{1}(B_R,\Sigma_k)$ such that  $\mathbf v_n=\u_n$ on $\partial B_R$ and $\mathbf v_n\to \mathbf 
    v$ in $H^{1}(B_R)$.
\end{Lemma}

\begin{proof}
Without loss of generality, we prove the statement for $R=1$. Let us fix a small parameter $\delta_n>0$ to be chosen later, and define
\[
A_n:=B_1\setminus B_{1-\delta_n} \quad \text{and} \quad t(r):=t_n(r)= \frac{1-r}{\delta_n} \quad \text{for }r\in[1-\delta_n,1].
\]
First, we define the competitor $\mf{v}_n$ in $B_{1-\delta_n}$, as
\[
\mf{v}_n(x)  :=\mathbf v\left(\frac{x}{1-\delta_n}\right) \qquad \text{for }x\in B_{1-\delta_n}.
\]
Next, we want to define $\mathbf v_n$ in the annulus $A_n$. We consider polar coordinates $(r,\theta)\in \R_+\times \mathbb S^{N-1}$, and let
\[
\mathbf g(\theta):= \u (1,\theta),\qquad
\text{ and }\qquad
\mathbf g_n(\theta) := \u_n(1,\theta).
\]
Fix $i\in\{1,\dots,k\}$ and $\theta \in \mathbb{S}^{N-1}$. We distinguish two cases: if $g_j(\theta)= g_{j,n}(\theta)=0$ for every $j\not=i$, we define for $r \in(1-\delta_n,1]$,
\[
\mathbf v_n(r,\theta):= \big( t(r) g_i(\theta)+(1-t(r)) g_{i,n}(\theta)\big) \, \mathbf e_i.
\]
Instead, if $g_i(\theta)>0$ and $g_{j,n}(\theta)>0$ for some $j\not=i$, we define for $r \in[1-\delta_n,1]$,
\[
\mathbf v_n(r,\theta):= \begin{cases}
  \displaystyle   \big(t(r) g_i(\theta)+(t(r)-1) g_{j,n}(\theta)\big)\,\mathbf{e}_i & \text{if } t(r) g_i(\theta)+(t(r)-1) g_{j,n}(\theta) \ge 0 \\
  \displaystyle   -\big(t(r) g_i(\theta)+(t(r)-1) g_{j,n}(\theta)\big)\,\mathbf{e}_j & \text{if } t(r) g_i(\theta)+(t(r)-1) g_{j,n}(\theta)<0.
\end{cases}
\]
We claim that this is the desired sequence, for an appropriate choice of $ \delta_n\to0$.

First, by construction, $\mathbf{v}_n \in H^1(B_1)$, and the segregation condition $v_{i,n}v_{j,n}=0$ for $i \neq j$ holds throughout $B_1$. Since $\mathbf{v}_n = \mathbf{u}_n$ on $\partial B_1$, it remains to prove that $\mathbf{v}_n \to \mathbf{v}$ in $H^1(B_1)$ as $n\to\infty$ by appropriately choosing $ \delta_n\to0$. For every $i=1,\dots,k$, we observe that
\[
\int_{B_{1-\delta_n}}|\D  v_{i,n}|^2\,dx = (1-\delta_n)^{N-2} \int_{B_{1}}|\D  v_i|^2\,dx \to \int_{B_{1}}|\D  v_i|^2\,dx\quad \text{ as }n \to \infty.
\]
Since in addition $\D v_{i,n} \to \D v_i$ a.e. in $B_1$, to conclude that $\mathbf{v}_n \to \mathbf{v}$ in $H^{1}(B_1)$ it suffices to show that $\|\nabla \mathbf{v}_n\|_{L^2(A_n)} \to 0$ as $n\to\infty$. For each $i=1,\dots,k$, using polar coordinates $(r,\theta) \in [1-\delta,1] \times \mathbb{S}^{N-1}$, we have
\begin{equation}\label{eq:construction:competitor:1}
    \int_{A_n} |\nabla v_{i,n}|^2 \, dx = \int_{\mathbb{S}^{N-1}} \int_{1-\delta_n}^{1} (\partial_r v_{i,n})^2 r^{N-1} \, dr d\theta + \int_{\mathbb{S}^{N-1}} \int_{1-\delta_n}^{1} \frac{|\nabla_\theta v_{i,n}|^2}{r^2} r^{N-1} \, dr d\theta=: \text{I} + \text{II}.
\end{equation}
The second term in the above identity vanishes due to assumption $[\u_n]_{C^{0,1}(\overline{B_1})}\le C$; indeed,
\[
\text{II}\le C \int_{\mathbb{S}^{N-1}} \int_{1-\delta_n}^{1} \sum_i \big(|\D_\theta g_i|^2 + |\D_\theta g_{i,n}|^2 \big)\, r^{N-3}dr d\theta \le C \delta_n \to 0 \quad \text{as }n \to \infty.
\]
Let us consider the sets $\mathcal C_i:=\{\theta\in\mathbb S^{N-1}:g_j(\theta)= g_{j,n}(\theta)=0 \, \text{ for } j\not=i \}$ and $\mathcal C_{i,j}:=\{\theta\in\mathbb S^{N-1}:g_i(\theta)>0,\, g_{j,n}(\theta)>0 \}$ for $i\not=j$.
Observing that $|t'(r)|= 1/\delta_n$, we can estimate the first term in \eqref{eq:construction:competitor:1} as follows:
\begin{align*}  \text{I}&=\int_{\mathbb{S}^{N-1}\cap \mathcal{C}_i} \int_{1-\delta_n}^{1} (\partial_r v_{i,n})^2 r^{N-1} \, dr d\theta
+\sum_{j\not=i}
\int_{\mathbb{S}^{N-1}\cap \mathcal{C}_{i,j}} \int_{1-\delta_n}^{1} (\partial_r v_{i,n})^2 r^{N-1} \, dr d\theta\\
&\le 
\int_{\mathbb{S}^{N-1}\cap \mathcal{C}_i} \int_{1-\delta_n}^{1} 
\frac{1}{\delta_n^2} |g_i(\theta)-g_{i,n}(\theta)|^2 r^{N-1}\, dr d\theta
+
\sum_{j\not=i}
\int_{\mathbb{S}^{N-1}\cap \mathcal{C}_{i,j}}\int_{1-\delta_n}^{1} 
 \frac{1}{\delta_n^2} |g_i(\theta)+g_{j,n}(\theta)|^2 r^{N-1}\, dr d\theta\\
&\le \frac{C}{\delta} \|g_i-g_{i,n}\|^2_{L^2(\mathbb S^{N-1})} + \frac{C}{\delta} \sum_{j\not=i} \|g_i+g_{j,n}\|^2_{L^2(\mathbb S^{N-1}\cap \mathcal{C}_{i,j})}\\
&\le \frac{C}{\delta} \left(
\|g_i-g_{i,n}\|^2_{L^2(\mathbb S^{N-1})}
+\sum_{j\not=i} 
\|g_i\|^2_{L^2(\mathbb S^{N-1}\cap \mathcal{C}_{i,j})} 
+
\sum_{j\not=i} \|g_{j,n}\|^2_{L^2(\mathbb S^{N-1}\cap \mathcal{C}_{i,j})} 
\right).
\end{align*}
By assumption $\u_n\to\u$ in $H^1(B_1)$ we have $\|\u_n- \u\|_{L^2(\partial B_1)}=:\e_n\to 0^+$. Hence, 
\[
\|g_i-g_{i,n}\|^2_{L^2(\mathbb S^{N-1})}\le \sum_i \|g_i-g_{i,n}\|^2_{L^2(\mathbb S^{N-1})}:=\e_n^2;
\]
moreover, since $g_{i,n}=0$ in $\mathcal{C}_{i,j}$ and $g_j=0$ when $g_i>0$, it follows that
\[
\|g_i\|^2_{L^2(\mathbb S^{N-1}\cap \mathcal{C}_{i,j})} = \|g_i-g_{i,n}\|^2_{L^2(\mathbb S^{N-1}\cap \mathcal{C}_{i,j})} \le \e_n^2, \qquad
\|g_{j,n}\|^2_{L^2(\mathbb S^{N-1}\cap \mathcal{C}_{i,j})} = \|g_j-g_{j,n}\|^2_{L^2(\mathbb S^{N-1}\cap \mathcal{C}_{i,j})} \le \e_n^2.
\]
Finally, choosing $\delta_n:={\e_n}$, we obtain $\text{I}\le C\e_n \to 0$ as $n\to\infty$ and the statement follows.
\end{proof}

\subsection{Blow-up analysis for \texorpdfstring{$\mathcal{O}(u,x_0)<2$}{}} The main result of this subsection is the following:

\begin{Proposition}\label{T:blow:up<2}
    Let $\u$ be a minimizer of \eqref{e:partition.functions}, $\Lambda=(\lambda_{2,1}(\omega_1),\dots,\lambda_{2,1}(\omega_k))$ and $x_0 \in Z(\u)$ be such that $\mathcal O(\u,x_0)<2$. Then, there exists a sequence $r_n\to 0^+$ such that 
    \[
    \u_{x_0,r_n}(x):=\frac{\u(x_0+r_nx)}{\Big(r_n^{1-N}H(\u,x_0,r_n)\Big)^{1/2}} \ \text{ converges to $\u_0\not\equiv0$ strongly in }H^1_{\loc}(\R^N) \text{ and in } C^{0,\alpha}_{\loc}(\R^N).
    \]
    Moreover, the limit $\u_0$ is a $\mathcal O(\u,x_0)$-homogeneous local minimizer of $J_{0,\cdot}$ in $\R^N$; in particular, $\u_0 \in \cG_{\loc}(\R^N)$ (see Definitions \ref{def J} and \ref{def G}).  
\end{Proposition}

A crucial step in the proof of the theorem consists in showing the local boundedness in $H^1$ of the sequence $\u_{x_0,r_n}$, which allows us to apply Lemma \ref{L:compactness} in order to deduce local convergence. Since $\mathcal{O}(u,x_0)<2$, to prove the boundedness we can adapt the argument of \cite[Section 3]{TaTe12}, based on the monotonicity of the frequency function. 

\begin{lemma}\label{L:bdd<2}
For every $R>0$, the sequence $\{\|\u_{x_0,r_n}\|_{0,R}\}$ is bounded.
\end{lemma}
\begin{proof}
    The fact that 
    \[
    \frac{1}{R^{N-1}}\int_{\pa B_R} |\u_{x_0,r_n}|^2\,d\sigma \le C,
    \]
    follows easily from the doubling property, Lemma \ref{lem: doub} (we refer to \cite[Lemma 3.5]{TaTe12} for more details). Concerning the boundedness of the gradient, we observe that
    \be\label{461}
    \begin{split}
        \frac{1}{R^{N-2}}\int_{B_R} & |\nabla \u_{x_0,r_n}|^2\,dx  = \frac{(r_nR)^{1-N}H(\u,x_0,r_nR)}{r_n^{1-N}H(\u,x_0,r_n)}\cdot\frac{r_nR \int_{B_{r_nR}(x_0)}|\nabla \u|^2\,dx}{\int_{\pa B_{r_nR}(x_0)} |\u|^2\,d\sigma} \\
        & \le C \left( N_1(\u,x_0,r_nR) + \frac{(r_nR)^2}{(r_nR)^{1-N} H(\u,x_0,r_nR)}\cdot \frac{1}{(r_nR)^{N}} \int_{B_{r_nR}(x_0)}(\Lambda\cdot \u)\,dx\right),
    \end{split}
    \ee
    where we used again the doubling property. Now, $N_1(\u,x_0,r_n R)\le C$ by monotonicity of the frequency, Lemma \ref{p:Almgren-below}. Regarding the other term on the right side, it is convenient to note first that arguing as in \eqref{2731} and \eqref{e:51}
    \[
    (r_nR)^2\|\u\|_{x_0,r_nR} \left( \frac{\|\u\|_{x_0,r_nR}}{(r_nR)^2} -C\right) \le \frac{1}{(r_nR)^{N-2}}D_1(\u,x_0,r_nR)
    +\frac{1}{(r_nR)^{N-1}}H(\u,x_0,r_nR),
    \]
    and
    \[
    \frac{\|\u\|_{x_0,r_nR}}{(r_nR)^2} -C \ge C(r_nR)^{-\delta'},
    \]
    for some constants independent of $n$, where $\delta' \in (0,2-\cO(\u,x_0))$. Therefore, by the Poincar\'e inequality and the boundedness of $N_1(\u,x_0,r_nR)$, we deduce that
    \begin{align*}
        \frac{(r_nR)^2}{(r_nR)^{1-N} H(\u,x_0,r_nR)} &\cdot \frac{1}{(r_nR)^{N}} \int_{B_{r_nR}(x_0)}(\Lambda\cdot \u)\,dx  \le \frac{(r_nR)^2}{(r_nR)^{1-N} H(\u,x_0,r_nR)} \|\u\|_{x_0,r_n R} \\
        & \le \frac{N_1(\u,x_0,r_nR)+1}{{(r_nR)^{-2}}\|\u\|_{x_0,r_nR} -C} \le C(N_1(\u,x_0,r_nR)+1) (r_nR)^{\delta'} \le C.
    \end{align*}
Coming back to \eqref{461}, the thesis follows.
\end{proof}

\begin{proof}[Proof of Proposition \ref{T:blow:up<2}]
Since $\mathcal O(\u,x_0)=\mathcal{V}(\u,x_0)<2$, along a suitable subsequence $r_n\to 0^+$ we have
\[
\lim_{n} \frac{H(\u,x_0,r_n)}{r_n^{N+3}}=+\infty.
\]
Moreover, by Lemma \ref{L:bdd<2}, the sequence $\|\u_{x_0,r_n}\|_{0,2R}$ is uniformly bounded in $n$, for every $R>1$. Hence, Lemma \ref{L:compactness} and a diagonal selection give the existence of a subsequence (still denoted by $r_n$) such that $\u_{x_0,r_n} \to \u_0$ in $H^1_{\loc}(\R^N)$ and in $C^{0,\alpha}_{\loc}(\R^N)$. The limit is non-trivial, since $\|\u_0\|_{L^2(\pa B_1)}=1$.

Let us now prove that, for every $R>1$, $\u_0$ is a minimizer of $J_{0,B_{R}}$. Let $\mf v \in H^1({B_{R},\Sigma_k})$ be such that $\mf v=\u_0$ on $\partial B_{R}$. By Lemma \ref{L:construction:commpetitor}, there exists a sequence $\mf v_n \in H^1({B_{R},\Sigma_k})$ such that $\mf v_n= \u_{x_0,r_n}$ on $\partial B_{R}$ and $\mf v_n \to \mf v$ in $H^1(B_{R})$. As observed in \eqref{eq:alm:min:rescaled}, the almost minimality condition rescaled to $ \u_{x_0,r_n}$ can be rewritten as
 \begin{multline*}
    \int_{B_{R}}|\nabla  \u_{x_0,r_n}|^2\,dx - 2\left(\frac{r_n^{N+3}}{H(\u,x_0,r_n)}\right)^{1/2} \int_{B_{R}}(\Lambda\cdot  \u_{x_0,r_n})\,dx\\
    \leq 
\int_{B_{R}}|\nabla \mf{v}_{n}|^2\,dx - 2\left(\frac{r_n^{N+3}}{H(\u,x_0,r_n)}\right)^{1/2} \int_{B_{R}}(\Lambda \cdot \mf{v}_{n})\,dx + C \frac{r_n^{N+3}}{H(\u,x_0,r_n)} r_n^{N-1}.
\end{multline*}
Passing to the limit as $n \to +\infty$, we obtain
\[
\int_{B_{R}}|\D \u_0|^2\,dx \le \int_{B_{R}}|\D \mf v|^2\,dx.
\]
Thus, $\u_0$ is a minimizer of $J_{0,B_R}$ and belongs to the class $\mathcal{G}(B_{R})$, for every $R>1$.
 
Finally, to prove that $\u_0$ is $\mathcal O(\u,x_0)$-homogeneous, it is enough to note that for every $\rho>0$ 
\[
        \frac{\displaystyle{\rho}
\int_{B_\rho}\left(
    |\D  \u_{x_0,r_n}|^2-\left(\frac{r_n^{N+3}}{H(\u,x_0,r_n)}\right)^{1/2} (\Lambda\cdot  \u_{x_0,r_n})
    \right)\,dx }{\displaystyle{\int_{\partial B_{\rho}}| \u_{x_0,r_n}|^2\,d\sigma} }= N_1(\u,x_0,r_n\rho),
\]
and take the limit as $n \to \infty$: we deduce that $N_0(\u_0,0,\rho)=N_1(\u,x_0,0^+) = \mathcal O(\u,x_0)$
for every $\rho >0$, where the last equality is a consequence of \eqref{ordini uguali}.
Therefore, Proposition \ref{prop: alm G} allows us to conclude.
\end{proof}

\subsection{Blow-up analysis for \texorpdfstring{$\mathcal{O}(u,x_0)=2$}{}}

The blow-up analysis when the vanishing order is smaller than $2$ is based on the validity of the monotonicity of the frequency, Lemmas \ref{p:Almgren-below} and \ref{lem: alm unif}. Results of this type are not available at the quadratic threshold $\cO(\u,x_0)=2$. As a consequence, we have to argue in a different way, exploiting the non-degeneracy and the monotonicity of the Weiss-type functionals. This strategy is inspired by \cite{SoaTer18}.

\begin{Proposition}\label{T:blowup=2:harmonic}
    Let $\u$ be a minimizer of \eqref{e:partition.functions}, $\Lambda=(\lambda_{2,1}(\omega_1),\dots,\lambda_{2,1}(\omega_k))$, and let $x_0 \in Z(\u)$ be such that $\cO(\u,x_0)= 2$. 
    Then, there exists a sequence $r_n\to 0^+$ such that 
    \[
     \u_{x_0,r_n}(x):=\frac{\u(x_0+r_nx)}{\|\u\|_{x_0,r_n}} \ \text{ converges to } \u_0\not\equiv0 \text{ strongly in }H^1_{\loc}(\R^N) \text{ and in }C^{0,\alpha}_{\loc}(\R^N).
    \]
    Moreover, the following alternative holds:
    \begin{enumerate}
    \item[($i$)] If 
    \[
    \limsup_{n} \frac{\|\u\|_{x_0,r_n}}{r_n^2}=+\infty,
    \]
    then $\u_0$ is a $2$-homogeneous local minimizer of $J_{0,\cdot}$ in $\R^N$ (see Definition \ref{def J}). 
    \item[($ii$)] If 
    \[
    \limsup_{n} \frac{\|\u\|_{x_0,r_n}}{r_n^2}<+\infty,
    \]
    then there exists a positive $\tilde\Lambda \in \R^k$ such that $\u_0$ is a $2$-homogeneous local minimizer of $J_{\tilde \Lambda,\cdot}$ in $\R^N$ (see Definition \ref{def J}).
    \end{enumerate}
\end{Proposition}

As in the previous section, we first show the uniform local boundedness in $H^1$.

\begin{lemma}\label{lem: bdd =2}
For every $R>0$, the sequence $\{\|\u_{x_0,r_n}\|_{0,R}\}$ is bounded.    
\end{lemma}

\begin{proof}
    Since $\|\u_{x_0,r_n}\|_{0,R} = \|\u\|_{x_0,R r_n}/\|\u\|_{x_0,r_n}$, we proceed by contradiction assuming that for a sequence $r_n \to 0^+$ 
\begin{equation}\label{30 11 1}
\frac{\|\u\|_{x_0,R r_n}}{ \|\u\|_{x_0,r_n}} \to +\infty \quad \text{as $n \to \infty$}.
\end{equation}
We claim that in such a case
\begin{equation}\label{30 11 2}
\frac{\|\u\|_{x_0,Rr_n}}{ (R r_n)^2} \to +\infty \quad \text{as $n \to \infty$}.
\end{equation}
If not, by non-degeneracy (Theorem \ref{prop: non-deg}), up to a subsequence we would have that 
\[
\|\u\|_{x_0,Rr_n} \le C (Rr_n)^{2} \le C R^{2} \|\u\|_{x_0,r_n},
\]
contradicting \eqref{30 11 1}. Thus \eqref{30 11 2} holds. Now, by the Poincar\'e inequality and \eqref{30 11 2} 
\[
\frac{1}{\|\u\|_{x_0,R r_n}^2}\cdot \frac1{(R r_n)^{N-2}} \int_{B_{R r_n}(x_0)} (\Lambda\cdot\u)\,dx \le \frac{C (R r_n)^2} {\|\u\|_{x_0,Rr_n}^2} \cdot \|\u\|_{x_0,Rr_n} \to 0 \qquad \text{as $n \to \infty$}.
\]
This estimate and the monotonicity of $W_{2,2}$ (Lemma \ref{L:monotonicity:W_2}) yield
\begin{align*}
C & \ge W_{2,2}(\u,x_0,R r_n)  \\
& = \frac{\|\u\|_{x_0, R r_n}^2}{(R r_n)^{4}}\left[1 - \frac{3 H(\u,x_0, R r_n)}{(R r_n)^{N-1}\|\u\|_{x_0, R r_n}^2 } -  \frac2{(R r_n)^{N-2} \|\u\|_{x_0,R r_n}^2} \int_{B_{R r_n}(x_0)} (\Lambda\cdot \u)\,dx\right] \\
& \ge  \frac{\|\u\|_{x_0, R r_n}^2}{(R r_n)^{4}} \left[ \frac34 -  \frac{3H(\u,x_0, R r_n)}{(R r_n)^{N-1} \|\u\|_{x_0, R r_n}^2}\right],
\end{align*}
for every $n$ large, which together with \eqref{30 11 2} implies that
\begin{equation}\label{H su norm non-deg}
\frac{H(\u,x_0, R r_n)}{(R r_n)^{N-1}\|\u\|_{x_0, R r_n}^2 } \ge \frac1{4} >0,
\end{equation}
for every $n$ large. 

We are ready to reach a contradiction. Let us consider the sequence $\{\u_{x_0,R r_n}\}$. By definition, it is bounded in $H^1(B_1)$, and hence by compactness of the Sobolev embedding of $H^1$ into $L^2$, and of the trace operator, up to a subsequence $\u_{R r_n} \to \bar \u$ weakly in $H^1(B_1)$, strongly in $L^2(B_1)$ and strongly in $L^2(\pa B_1)$. Actually, since in the present setting the assumptions of Lemma \ref{L:compactness} are satisfied (by \eqref{30 11 2}), the convergence is also strong in $H^1_{\loc}(B_1)$ and in $C^{0,\alpha}_{\loc}(B_1)$. We claim that $\bar \u$ is a minimizer of $J_{0,B_\rho}$ for every $\rho \in (0,1)$. This can be proved exactly as in the proof of Proposition \ref{T:blow:up<2}, by passing to the limit in the scaled almost minimality condition \eqref{eq:alm:min:rescaled} and using \eqref{30 11 2}. Now, on the one hand, by estimate \eqref{H su norm non-deg}
\[
H(\bar \u,0,1) = \lim_{n } H(\u_{x_0,R r_n}, 0,1) = \lim_{n} \frac{H(\u,x_0, R r_n)}{(R r_n)^{N-1}\|\u\|_{x_0, R r_n}^2 } \ge C>0,
\]
so that $\bar \u \not \equiv 0$ in $B_1$, and by minimality $\bar \u \in \cG_{\loc}(B_1)$. But on the other hand, having assumed \eqref{30 11 2} we also deduce that
\[
\|\bar \u\|_{0,1/R} = \lim_{n} \|\u_{R r_n}\|_{0,1/R} = \lim_{n} \frac{\|\u\|_{x_0,r_n}}{\|\u\|_{x_0,R r_n}} = 0,
\]
which forces $\bar \u \equiv 0$ in $B_{1/R}$. Since $\bar \u \in \cG(B_{\rho})$ for every $\rho \in (1/R,1)$, and the unique continuation property holds for functions in $\cG(B_{\rho})$ (see \cite[Remark 2.5]{TaTe12}), it is necessary that $\bar \u \equiv 0$ in $B_\rho$ for every $\rho \in (1/R,1)$; namely $\bar \u \equiv 0$ in $B_1$, which gives the desired contradiction.
\end{proof}

\begin{proof}[Proof of Proposition \ref{T:blowup=2:harmonic}]
    \emph{Case ($i$)}. Let $r_n$ be such that $\alpha_n:= \|\u\|_{x_0,r_n}/r_n^2 \to +\infty$. By Lemmas \ref{L:compactness} and \ref{lem: bdd =2}, up to a subsequence and a diagonal selection we have convergence $\u_{x_0,r_n} \to \u_0$ in $H^1_{\loc}(\R^N)$ and in $C^{0,\alpha}_{\loc}(\R^N)$, with $\u_0 \not \equiv 0$ as $\|\u_0\|_{0,1}=1$. As in the proof of Proposition \ref{T:blow:up<2}, by passing to the limit in the scaled almost minimality condition \eqref{eq:alm:min:rescaled} and using the fact that $\alpha_n \to +\infty$ we infer that $\u_0$ is a minimizer of $J_{0,B_R}$ with fixed trace, for every $R>0$, and hence $\u_0 \in \cG_{\loc}(\R^N)$. It remains to show that $\u_0$ is $2$-homogeneous. To this end, we first observe that by monotonicity
    \[
    W_{2,2}(\u_{x_0,r_n},0,t) = \frac{r_n^4}{\|\u\|_{x_0,r_n}^2} W_{2,2}(\u,x_0,r_nt) \le \frac{r^4}{\|\u\|_{x_0,r_n}^2} W_{2,2}(\u,x_0,r_0).
    \]
    Since $\u \in H^1(B_{r_0})$, this implies that
    \[
    \frac{1}{t^{N-2}}\int_{B_t}|\nabla \u_{x_0,r_n}|^2\,dx \le C\frac{t^2}{\alpha_n^2} + \frac{1}{\alpha_n t^N} \int_{B_t} (\Lambda\cdot \u_{x_0,r_n})\,dx + \frac{2}{t^{N-1}} \int_{\pa B_t} |\u_{x_0,r_n}|^2\,d\sigma,
    \]
    and passing to the limit we deduce that
    \[
    \frac{1}{t^{N-2}}\int_{B_t}|\nabla \u_{0}|^2\,dx \le \frac{2}{t^{N-1}} \int_{\pa B_t} |\u_{0}|^2\,d\sigma,
    \]
    for every $t>0$. Recalling that $\u_0$ is an element of $\cG_{\loc}(\R^N)$, we can rewrite this condition in terms of the frequency as $N(\u_0,0,t) \le 2$ for every $t$. On the other hand, Proposition \ref{prop: upper} ensures that $\mathcal{O}(\u_0,0) \ge 2$, namely $N( \u_0,0,0^+) \ge 2$, which by monotonicity gives $N(\u_0,0,t) \ge 2$ for every $t>0$. Therefore, the frequency function is constant, equal to $2$, and hence $\u_0$ is a quadratic homogeneous function in $\cG_{\loc}(\R^N)$ (see Proposition \ref{prop: alm G}). 

    \emph{Case ($ii$)}. Let $r_n$ be such that $\alpha_n:= \|\u\|_{x_0,r_n}/r_n^2 \to \ell$, with $\ell \in (0,+\infty)$ by assumption and non-degeneracy, Theorem \ref{prop: non-deg}. As in case ($i$), up to a subsequence and a diagonal selection we have convergence $\u_{x_0,r_n} \to \u_0$ in $H^1_{\loc}(\R^N)$ and in $C^{0,\alpha}_{\loc}(\R^N)$, with $\u_0 \not \equiv 0$. Moreover, $\u_0$ is a minimizer of $J_{\tilde \Lambda,B_R}$ with fixed trace, for every $R>0$, where $\tilde \Lambda:= \Lambda/\ell$ is a positive vector. Concerning the homogeneity, for every $t>0$
    \[
    W_{2,2}(\u_{x_0,r_n},0,t) = \frac{r_n^4}{\|\u\|_{x_0,r_n}^2} W_{2,2}(\u,x_0,r_nt) \quad \implies \quad W_{2,2}(\u_{0},0,t) = \frac{1}{\ell^2}W_{2,2}(\u,x_0,0^+),
    \]
    where the last limit exists by Lemma \ref{L:monotonicity:W_2}. Therefore, the Weiss-type functional $W_{2,2}$ of $\u_0$ is constant, and the homogeneity follows.
\end{proof}

\section{Regularity of the nodal set}\label{sec: dim}

In this section, we provide the proof of Theorem \ref{thm: fb}. For the sake of clarity, we divide the proof into two parts: the first addresses the dimensional estimates of the nodal and singular sets, while the second focuses on the regularity of the regular part of the nodal set.

\subsection{Hausdorff dimension of the nodal and singular sets}

We recall that the singular set is defined as
\[\Sigma(\u) = \left\{ x_0 \in Z(\u) \ \left| \ \mathcal{O}(\u, x_0) > 1 \right.\right\}.
\]
As a consequence of the results in Sections \ref{sec: upper} and \ref{sec: blow-up}, we obtain the dichotomy $\mathcal{O}(\u, x_0) = 1$ or $\mathcal{O}(\u, x_0) \ge 3/2$ for every $x_0 \in Z(\u)$, which implies that $\Sigma(\u)$ is a relatively closed subset of $Z(\u)$. Next, by applying Federer's dimensional reduction principle, we provide sharp estimates for the Hausdorff dimension of the nodal set $Z(\u)$ and the singular set $\Sigma(\u)$.

\begin{proof}[Proof of Theorem \ref{thm: fb},($a$)]

    \emph{Step 1) Either $\cO(\u,x_0) =1$ or $\cO(\u,x_0) \ge 3/2$}.
    By Theorem \ref{prop: non-deg}, we know that $\cO(\u,x_0) \in [1,2]$. Suppose that $\cO(\u,x_0) <2$ (otherwise there is nothing to prove). Then, up to a subsequence $\u_{x_0,r} \to \u_0$ as $r \to 0^+$, with $\u_0 \in \cG_{\loc}(\R^N)$ homogeneous of degree $\cO(\u,x_0)$. As a consequence, $\cO(\u,x_0)=N(\u_0,0,0^+)$, and the claim follows from Proposition \ref{prop: alm G}-($iii$).

    \emph{Step 2) $\Sigma(\u)$ is a relatively closed subset of $Z(\u)$.} Let $\{x_n\} \subset \Sigma(\u)$ be a converging sequence: $x_n \to x_0 \in \Omega$. Let us consider $\limsup_n \cO(\u,x_n) \ge 3/2$, by Step 1. If $\limsup_{n} \mathcal{O}(\u,x_n) <2$, the estimate $\cO(\u,x_0) \ge 3/2$ follows from Proposition \ref{prop: upper}-($ii$). If instead $\limsup_{n} \mathcal{O}(\u,x_n) \ge 2$, the fact that $\cO(\u,x_0) \ge 2>3/2$ follows from Proposition \ref{prop: upper}-($i$). In both cases, $x_0 \in \Sigma(\u)$.

    \emph{Step 3) Dimension of the nodal set.} 
    As observed in \cite[Theorem 4.5]{TaTe12}, we only need to prove that the Hausdorff
dimension estimates hold for $Z(\u) \cap B_1$ and $\Sigma(\u) \cap B_1$, whenever $\u$ is a minimizer of $J_{\Lambda,\O}$ with a positive $\Lambda \in \R^k$ and $\Omega \supset \supset B_2$. To this end, let 
    \[
    \mathcal{F}:=\left\{ \mathbf{v} \in L^\infty_{\loc}(\R^N,\R^k) \left|\begin{array}{l} \text{$\mathbf{v} \gneqq 0$ and there exist a domain $\O \supset\supset B_2$ such that }\\
    \text{$\mathbf{v}$ is either a minimizer of $J_{\lambda,\O}$ for a positive $\lambda \in \R^k$,}\\
    \text{or a minimizer of $J_{0,\O}$.}\end{array}\right.\right\}, \]
    endowed with the $C^{0,\alpha}(\overline{B_2})$ and strong $H^1(B_2)$ convergence. We apply Federer's dimensional reduction principle exactly as stated in \cite[Theorem 4.6]{TaTe12}. The family $\mathcal{F}$ is closed under scaling and translations. Moreover, there exist homogeneous blow-up limits within the class $\mathcal{F}$ at every nodal point. For every $\mathbf{v}$, let $\mathcal{S}(\mf{v}):= Z(\mf{v}) \cap B_1$, which is a closed subset of $B_1$. The local uniform convergence of the blow-up sequences allows us to easily check the validity of the so-called ``singular set assumptions" (assumption (A3) in \cite[Theorem 4.6]{TaTe12}). Therefore, Federer's principle and the fact that $\mf{v} \equiv 0$ does not belong to the class $\mathcal{F}$ imply that $\mathrm{dim}_{\mathcal{H}}(Z(\u))  \le N-1$.
    
     \emph{Step 4) Dimension of the singular set.} With $\mathcal{F}$ as in the previous step, let $\mathcal{S}^*(\mf{v}):= \Sigma(\mf{v}) \cap B_1$, which is a closed subset of $B_1$ by Step 2. It is clear that 
     \[
     \Sigma\left(\frac{\mf{v}(x_0+r\,\cdot)}{\rho_r}\right) = \frac{\Sigma(\mf{v})-x_0}{r}.
     \]
     Thus, to check that the singular set assumption is satisfied, it remains to show that if a blow-up sequence $\mf{v}_{x_0,r_n}$ (defined as in Section \ref{sec: blow-up}) is convergent to $\bar{\mf{v}}$, then
     \[
     \forall \eps>0 \ \exists \bar n: \ n>\bar n \ \implies \ \mathcal{S}^*(\mf{v}_{x_0,r_n}) \subset \{x \in \R^N \ \left| \ \dist(x,\mathcal{S}^*(\bar{\mf{v}}))<\eps\right.\}.
     \]
     By contradiction, we assume that there exists $\bar \eps>0$ and a sequence $\{x_n\} \subset B_1$ such that $x_n \in \mathcal{S}^*(\mf{v}_{x_0,r_n})$ for every $n$, and $\dist(x_n, \mathcal{S}^*(\bar{\mf{v}})) \ge \bar \eps$ along a sequence of indices $n\to\infty$. Up to a subsequence $x_n \to \bar x \in \overline{B_1}$. Since $\cO(\mf{v}_{x_0,r_n},x_n) \ge 3/2$ for every $n$ by Step 1, Proposition \ref{prop: upper} implies that $\cO(\bar{\mf{v}},\bar x) \ge 3/2$ as well, so that $\dist(x_n, \mathcal{S}^*(\bar{\mf{v}})) \to0$, which is the desired contradiction.

     This means that Federer's principle is applicable, and implies that there exists an integer $0 \le d \le N-1$ such that 
     $\mathrm{dim}_{\mathcal{H}}(\mathcal{S}^*(\mf{v}))  \le d$ for every $\mf{v} \in \mathcal{F}$. Moreover, there exist a $d$-dimensional subspace $L \subset \R^N$ and a function $\mf{v} \in \mathcal{F}$ (for some $\alpha \ge 0$) such that $\Sigma(\mf{v}) = L$ and $\mf{v}$ is $\alpha$-homogeneous with respect to any $x_0 \in L$, namely
     \[
     \frac{\mf{v}(x_0+\lambda x)}{\lambda^\alpha} =\mf{v}(x) \quad \text{for all } x_0 \in L \ \text{and} \ \lambda>0.
     \]
     Hence, if we suppose by contradiction that $d=N-1$, up to a rotation we have that $L=\R^{N-1} \times \{0\}$ and that $\mf{v}$ depends only on $x_N$, $\mf{v}(x) = \mf{w}(x_N)$. To reach a contradiction, it is sufficient to show that there is no $\mf{v}$ satisfying all these properties. If $\mf{v}$ is a minimizer of $J_{0,B_2}$, then $\mf{v} \in \cG(B_2)$. If $w_i(1)=0$ (resp. $w_i(-1)=0$), then by homogeneity $v_i \equiv 0$ in $\R^{N-1} \times [0,+\infty)$ (resp. $v_i \equiv 0$ in $\R^{N-1} \times (-\infty,0]$). Since $Z(\mf{v})$ has empty interior, it is necessary that one component, say $v_1$, satisfies $w_1(1)>0$, and another component, say $v_2$, satisfies $w_2(-1)>0$. By homogeneity and by the segregation condition, $\{w_1>0\}  =(0,+\infty)$ and $\{w_2>0\}=(-\infty,0)$.
     Moreover, since $\mf{v} \in \cG(B_2)$, $w_1$ is harmonic on $(0,+\infty)$ and $w_2$ is harmonic on $(-\infty,0)$. Thus, both $w_1$ and $w_2$ are linear and positive in $(0,+\infty)$ and $(-\infty,0)$, respectively, so that $\cO(\mf{v},(x',0))=1$, in contradiction with the fact that $\Sigma(\mf{v}) = \R^{N-1} \times \{0\}$.
     
   If $\mf{v} \in \mathcal{F}$ is a homogeneous minimizer of $J_{\lambda,\Omega}$ for a positive $\lambda$, we can argue exactly as before. The only difference is that $\mf{v} \in \mathcal{T}_\lambda(B_2)$, and hence $w_1''=-\lambda_1$ in $(0,+\infty)$ and $w_2''=-\lambda_2$ in $(-\infty,0)$. By homogeneity, we deduce that $w_1(x_N) = - \lambda_1 x_N^2/2$ and $w_2(x_N) = -\lambda_2 x_N^2/2$, which on the other hand is not possible since $\mf{v} \ge 0$. This contradiction finally gives the desired estimate. 
\end{proof}

\subsection{Regular part of the nodal set}

In this section we prove that the regular nodal set $\mathcal R(\u) = Z(\u) \setminus \Sigma(\u)$ is locally a regular $(N-1)$-dimensional hypersurface. Since $\mathcal R(\u)=\{x \in Z(\u) \ \left| \ \cO(\u,x)=1\right.\}$, and at point of vanishing order $1$ we have a uniform Almgren-type monotonicity formula, we can adapt the argument used in \cite{TaTe12}. We shall consider blow-up sequences with varying centers $x_n$, under the assumption that $\mathcal O(\u,x_n)=1$ for every $n$, $x_n \to x_0$, and $\mathcal O(\u,x_0)=1$. 

\begin{Proposition}\label{T:blow:up=1}
    Let $\u$ be a minimizer of \eqref{e:partition.functions} and $\Lambda=(\lambda_{2,1}(\omega_1),\dots,\lambda_{2,1}(\omega_k))$. Let $\{x_n\} \subset Z(\u)$ be such that $\mathcal O(\u,x_n)=1$ for every $n$, $x_n \to x_0$ and $\mathcal O(\u,x_0)=1$. Then, there exists a sequence $r_n\to 0^+$ such that 
    \[
    \u_{x_n,r_n}(x):=\frac{\u(x_n+r_nx)}{\Big(r_n^{1-N}H(\u,x_n,r_n)\Big)^{1/2}}
 \ \text{ converges to $\u_0 \not \equiv 0$ strongly in }H^1_{\loc}(\R^N) \text{ and in }C^{0,\alpha}_{\loc}(\R^N).
    \]
    Moreover, the limit $\u_0$ is a $1$-homogeneous minimizer of $J_{0,\cdot}$ in $\R^N$; in particular, $\u_0 \in \cG_{\loc}(\R^N)$ and hence, up to a relabelling of the components and to a rotation,
    \be\label{1-hom in G}
    \u_0(x) = c(x_N^+, x_N^-, 0,\dots,0),
    \ee
    for some constant $c>0$.
\end{Proposition}

\begin{proof}
    We proceed exactly as in Proposition \ref{T:blow:up<2}, but we use the uniform-in-$x$ monotonicity formula and doubling property, Lemmas \ref{lem: alm unif} and \ref{lem: doub unif}, in place of Lemmas \ref{p:Almgren-below} and \ref{lem: doub}. For this purpose, we stress that by convergence and recalling that $\mathcal R(\u)$ is relatively open in $Z(\u)$, the points $x_0$ and $x_n$ belong to a common compact set $K$ at positive distance from $\Sigma(\u)$. Finally, the fact that $1$-homogeneous elements of $\cG_{\loc}(\R^N)$ are of type \eqref{1-hom in G} follows from the definition, see \cite{TaTe12} for the details.
\end{proof}

To prove the regularity of $\mathcal{R}(\u)$, we rely on the approach introduced in \cite[Section 5]{TaTe12}. We include a brief sketch of the argument here for the sake of completeness, emphasizing the main differences specific to our case.

\begin{proof}[Proof of Theorem \ref{thm: fb},($b$)]

 \emph{Step 1)}  Let us take an open set $\tilde\Omega\subset\Omega$ such that $\mathcal R(\u)\cap \overline{\tilde\Omega} =Z(\u)\cap \overline{\tilde\Omega}$. Using the compactness of the blow-up sequence with varying centers $x_n$ established in Proposition \ref{T:blow:up=1}, one can argue exactly as in \cite[Lemma 5.3]{TaTe12} to prove that $\mathcal R(\u)\cap \tilde\Omega$ satisfies the $(N-1)$-dimensional $(\delta,R)$-Reifenberg flat condition: for any given $0<\delta<1$ there exists $R>0$ such that for every $x\in \mathcal R(\u)\cap \tilde\Omega$ and $0<r<R$ there exists a hyper-plane $H_{x,r}$ containing $x$ such that 
 \[
 d_\mathcal H (\mathcal R(\u)\cap B_r(x), H_{x,r}\cap B_r(x) )\le \delta r,
 \]
 where $d_\mathcal H(\cdot,\cdot)$ denotes the Hausdorff distance.

 \emph{Step 2)}  As a consequence of the previous step, we prove a clean-up property arguing exactly as in \cite[Lemma 5.4]{TaTe12}: given $x_0\in \mathcal R(\u)$, there exists a radius $R_0>0$ such that $\mathcal R(\u)\cap B_{R_0}(x_0) = Z(\u)\cap B_{R_0}(x_0)$, and $B_{R_0}(x_0)\setminus \mathcal R(\u)$ has exactly two connected components $\O_1, \O_2$.

 \emph{Step 3)} We show that there exist two indices $i\neq j$ such that $u_i>0$ in $\O_1$ and $u_j>0$ in $\O_2$. If not, we would have $u_i>0$ in $\O_1 \cup \O_2$ and $u_j\equiv 0$ in $B_{R_0}(x_0)$ for every $j\neq i$. Using Proposition \ref{p:classeS}, it readily follows that
  \[
  -\Delta u_i = \lambda_{2,1}(\omega_i)\ind_{\{u_i>0\}} \ge 0 \quad \text{in } B_{R_0}(x_0).
  \]
  By the strong maximum principle, it follows that $u_i>0$ in $B_{R_0}(x_0)$, which contradicts $x_0\in Z(\u)$. 

 \emph{Step 4)} Using again Proposition \ref{p:classeS}, we see that the function $w:=u_i-u_j$ solves the \emph{unstable two-phase obstacle problem}
 \[
 -\Delta w = \lambda_{2,1}(\omega_i) \ind_{\{w>0\}} -\lambda_{2,1}(\omega_j) \ind_{\{w<0\}}\quad \text{in }B_{R_0}(x_0).
 \]
 By classical Schauder theory, $w \in C^{1,\alpha}_\loc(B_{R_0}(x_0))$ for every $\alpha\in(0,1)$. Arguing as in \cite[Corollary 5.8]{TaTe12} and using that $\mathcal{O}(\u,x_0)=1$ combined with Proposition \ref{T:blow:up=1}, we obtain that $\nabla w(x_0)\neq 0$. Consequently, the Implicit Function Theorem ensures that $\mathcal R(\u)\cap B_{R_0}(x_0)$ is locally a $C^{1,\alpha}$ hypersurface, completing the proof.
\end{proof}

\section{Homogeneous solutions in dimension \texorpdfstring{$N=2$}{}}\label{sec: hom}

The blow-up analysis reveals that at any nodal point a minimizer $\u=(u_1,\dots,u_k)$ of \eqref{e:partition.functions} can look either as a $\gamma$-homogeneous local minimizer of $J_{0,\cdot}$ in $\R^N$, with the additional condition $\gamma \in \{1\} \cup [3/2, 2]$, or as a $2$-homogeneous minimizer of $J_{\lambda,\cdot}$ in $\R^N$, for some positive $\lambda \in \R^k$. Thus, it is natural to wonder whether we have explicit examples of such configurations with $k$ non-trivial components in at least one of these classes, for every integer $k \ge 2$. Here we give a partial answer to this question in dimension $N=2$ (and hence in any dimension).

\subsection{Extremality conditions for \texorpdfstring{$2$}{2}-homogeneous local minimizers of \texorpdfstring{$J_{\lambda,\cdot}$}{}} Let $\lambda \in \R^k$ be positive. If $\u$ is a $2$-homogeneous local minimizer of $J_{\lambda,\cdot}$ in $\R^2$, then as before the following extremality conditions are satisfied:
\begin{enumerate}
    \item [($i$)]  $\u \in \mathcal{T}_{\lambda,\loc}(\R^2)$;
    \item [($ii$)] $Z(\u)$ has empty interior;
    \item [($iii$)] the reflection law holds at regular points of $Z(\u)$:
    given $x_0 \in \mathcal R(\u)$ and denoting by $\nu$ the normal vector to $Z(\u)$ at $x_0$, it holds \[
    \lim_{\substack{x\to x_0\\ x\in \{u_i>0\}} }\D u_i(x)\cdot \nu = -\lim_{\substack{x\to x_0\\ x\in \{u_j>0\}}}  \D u_j(x)\cdot \nu.
    \]
\end{enumerate}
For point ($i$), we refer to Remark \ref{rem: min G T}. Point ($ii$) is a consequence of the fact that local minimizers are non-degenerate; this can be proved as in Theorem \ref{prop: non-deg} (actually, the proof for local minimizers is simpler since we do not have to deal with the normalization condition on the $L^1$ norm). Point ($iii$) can be proved as in Theorem \ref{thm: fb},$(b)$.

A slightly weaker question than the existence of $2$-homogeneous local minimizers of $J_{\lambda,\cdot}$ in $\R^2$ consists in discussing the existence of $\gamma$-homogeneous functions (with $\gamma \in \{1\} \cup [3/2,2]$) satisfying the previous extremality conditions. In what follows we show that, for every $k \ge 3$, there exist multiple $2$-homogeneous admissible profiles with $k$ non-trivial components and at least $5$ connected components.

\begin{proposition}\label{prop: homog}
    Let $k \ge 2$ and $\lambda \in \R^k$ be positive. Then there exist infinitely many $2$-homogeneous functions in $H^1_{\loc}(\R^N,\Sigma_k)$ satisfying the above conditions ($i$)-($iii$). 
\end{proposition}

\begin{proof}
In case $k=2$, it is sufficient to consider a $2$-homogeneous sign-changing solution of 
\[
-\Delta w = \lambda_1 \ind_{\{w>0\}}-\lambda_2  \ind_{\{w<0\}} \qquad \text{in $\R^2$}, 
\]
and set $(u_1,u_2) = (w^+,w^-)$. As proved in \cite{SoaTer18}, there exist infinitely many such solutions, any solution is $C^{1,\alpha}$ regular, and its nodal set has empty interior. This implies that $(u_1,u_2)$ satisfies ($i$)-($iii$) above.

Thus, we can focus on the case $k \ge 3$. By definition and direct computations, $\u$ is $2$-homogeneous and satisfies ($i$)-($iii$) if: 
\begin{enumerate}
    \item [($a$)] $\u(r,\theta) =r^2 (\varphi_1(\theta),\dots,\varphi_k(\theta))$, where $\varphi_i$ are such that: if  $(a_i,b_i) \subset [0,2\pi)=\R/2\pi\Z$ is a positivity interval of $\varphi_i$, then
\be\label{761}
\begin{cases}
    \varphi_i''+4\varphi_i =-\lambda_i, \quad \varphi_i >0 & \text{in $(a_i,b_i)$}\\
    \varphi_i(a_i)=0=\varphi_i(b_i).
\end{cases}
\ee
    \item [($b$)] $\{\varphi_i>0\} \cap \{\varphi_j>0\} = \emptyset$ and $\bigcup_i \overline{\{\varphi_i>0\}}=\R/2\pi\Z$.
    \item [($c$)] If $(a_i,b_i)$ and $(a_j,b_j)$ are positivity intervals of $\varphi_i$ and $\varphi_j$, respectively, with $i \neq j$, and $b_i=a_j$, then $-\varphi_i'(b_i^-) =\varphi_j'(a_j^+)$.
\end{enumerate}
Problem \eqref{761} has the explicit solution
\[
\varphi_i(\theta) = \frac{\lambda_i}{4}\left(\frac{\cos(2\theta-a_i-b_i)}{\cos(b_i-a_i)}-1\right),
\]
which is positive only when $0<b_i-a_i<\pi/2$. 

Let $m \ge 1$ be an integer such that $k \pi/2>2\pi/m$. We define $\u$ in the circular sector $\{r\ge 0, \theta \in [0,2\pi/m]\}$ so that $\{\varphi_1>0\}=(0,a_1)$, $\{\varphi_2>0\}=(a_1,a_2)$, \ldots, $\{\varphi_k>0\} = (a_{k-1},2\pi/m)$, where $0=a_0<a_1<a_2<\dots<a_{k-1}<a_k=2\pi/m$ are determined by condition ($c$) above, and satisfy the constraint $0<a_i-a_{i-1}<\pi/2$. Introducing the new unknowns $A_i=a_i-a_{i-1} \in (0,\pi/2)$, simple computations yield the system
\be\label{860}
\lambda_i \tan A_i = const. = \alpha \quad \iff \quad A_i = \arctan\left(\frac{\alpha}{\lambda_i}\right),
\ee
where the value $\alpha$ must satisfy the scalar equation
\be\label{861}
\sum_{i=1}^k \arctan\left(\frac{\alpha}{\lambda_i}\right) = \sum_{i=1}^k A_i = \frac{2\pi}{m}.
\ee
Thus, our construction is complete if we can find $\alpha>0$ satisfying \eqref{861}. Denoting by $\Phi(\alpha)$ the left hand side in \eqref{861}, we observe that $\Phi$ is strictly increasing in $(0,+\infty)$, with $\Phi(0^+) = 0$ and $\Phi(+\infty) = k\pi/2$; since by assumption $k\pi/2>2\pi/m$, there exists a unique $\alpha>0$ solving \eqref{861}, and hence a uniquely determined $\u$ in the circular sector $\{r\ge 0, \theta \in [0,2\pi/m]\}$ satisfying the previous properties. At this point, we define $\u$ in the whole plane by periodicity: 
\[
\u(r,\theta)=\u\left(r,\theta-\frac{2\pi j}{m}\right),\quad \text{ for } j=1,\dots,m \text{ and } \theta\in\left[\frac{2\pi j}{m},\frac{2\pi (j+1)}{m}\right].
\]
In doing this, we have to check that this extension still satisfies ($a$)-($c$) above (in particular, $\varphi_1(0^+) = \varphi_k(2\pi/m^-)$). This is already included in \eqref{860}, so that $\u$ satisfies all the extremality conditions.

Since this construction can be performed for any $m \ge 1$ such that $k\pi/2>2\pi/m$, and there are infinitely many such integers, the proof is complete.
\end{proof}

\subsection{Extremality conditions for \texorpdfstring{$\gamma$}{2}-homogeneous local minimizers of \texorpdfstring{$J_{0,\cdot}$}{}}\label{sec: hom harm} Assuming that $\u$ is a $\gamma$-homogeneous local minimizer of $J_{0,\cdot}$ in $\R^2$, then the following extremality conditions are satisfied:
\begin{enumerate}
    \item [($i$)]$\u \in \cG_{\loc}(\R^2)$;
    \item [($ii$)] $Z(\u)$ has empty interior;
    \item [($iii$)] the reflection law holds at regular points of $Z(\u)$.
\end{enumerate}
Indeed, the fact that local minimizers belong to $\cG_{\loc}(\R^2)$ has already been observed in Remark \ref{rem: min G T}; and at this point both ($ii$) and ($iii$) are consequences of ($i$), as proved in \cite{TaTe12}. 

Such extremal configurations can be classified completely in dimension $N=2$ (we refer to \cite{TaTe12} and \cite[Theorem 1.4]{SoTe15} for more details). Clearly, the linear profiles $(x_2^+,x_2^-,0,\dots,0)$ and $(|x_2|,0,\dots,0)$ are admissible $1$-homogeneous profiles in $\cG_{\loc}(\R^2)$ (however, by minimality, it is not difficult to see that the only admissible configuration arising as a blow-up profile in Proposition \ref{T:blow:up<2} is the former one). Up to rotations, multiplications by a positive quantity, and relabelling of the components, these are the only linear profiles in $\cG_{\loc}(\R^2)$. A $3/2$-homogeneous function in $\cG_{\loc}(\R^2)$, with exactly 3 non-trivial components, is represented by
\[
u_1(r,\theta) = r^{3/2}\sin\left(\frac{3}{2}\theta\right) \ind_{\{0<\theta<2\pi/3\}}(\theta), \quad u_2(r,\theta) = u_1\left(r,\theta-\frac{2\pi}{3}\right), \quad u_3(r,\theta) = u_1\left(r,\theta-\frac{4\pi}{3}\right),
\]
with $u_j \equiv 0$ for $j \neq 1,2,3$; finally, we have some admissible $2$-homogeneous configurations in $\cG_{\loc}(\R^2)$, such that if $u_i \not \equiv  0$, then each connected component of its positivity set is a cone of opening $\pi/2$, say $\{r>0, \theta \in (0,\pi/2)\}$, and $u_i(r,\theta) = r^2 \sin(2\theta)$. Any such configuration has at most $4$ non-trivial components. 

If $k \ge 5$, then there is no $\gamma$-homogeneous extremal profile in $\cG_{\loc}(\R^2)$ with $\gamma < 5/2$ and at least five non-trivial components.

\begin{remark}
    The preceding discussion reveals a fundamental difference between the configurations
that may arise as blow-up limits when the limiting profile belongs to
$\mathcal{G}_{\mathrm{loc}}(\mathbb{R}^N)$ and those arising when the limiting
profile belongs to $\mathcal{T}_{\lambda,\mathrm{loc}}(\mathbb{R}^N)$ for some
positive $\lambda$, at least in dimension $N=2$. Indeed, in the former case
the profiles are essentially classifiable and, thanks to the bound on the degree
of homogeneity, they have at most four non-trivial components; in the latter case,
instead, there exist infinitely many extremal profiles with $k$ non-trivial
components, for every $k$.
\end{remark}

\section{Connectedness of the optimal partition in Problem B}\label{sec: conn}

In this last section, we prove Proposition \ref{prop: connect}. In particular, given a minimizer $\u=(u_1,\dots,u_k)$ of \eqref{pbB}, under the condition $\lambda_1=\dots=\lambda_k=\bar \lambda>0$ and assuming that $\{\varphi_i>0\}\subset \pa \Omega$ is connected for every $i$, then $\omega_i:=\{u_i>0\}$ is connected for every $i$. 

\begin{proof}[Proof of Proposition \ref{prop: connect}] 
By contradiction, let us assume that $\omega_1=\{u_1>0\} = \omega^* \cup \omega^{**}$, where $\omega^*$ and $\omega^{**}$ are two non-empty disjoint open sets.

\emph{Case 1)} Assume that $A_1:=\operatorname{int}(\overline{\omega^*\cup \omega^{**}} )$ is connected, and let $\bar u_1$ be the minimizer of the functional 
\be\label{eq:conn:I}
I_{A_1}(v) := \int_{A_1} (|\nabla v|^2 - 2 \bar \lambda v) \, dx, \text{ for } v \in H^{1}(A_1),\,  v = \varphi_1 \text{ on } \partial \Omega \cap \overline{A_1} \text{ and } v=0 \text{ on } \partial  A_1\setminus (\partial \Omega \cap \overline{A_1}).
\ee
Any minimizer is non-negative, and actually strictly positive by the strong maximum principle. Thus, in particular, $u_1$ is not a minimizer, and
\[
\int_{A_1} (|\nabla \bar u_1|^2 - 2\bar\lambda \bar u_1) \, dx < \int_{A_1} (|\nabla u_1|^2 - 2\bar\lambda u_1) \, dx.
\]
Consequently, the competitor $(\bar u_1, u_2, \dots, u_k)$ in the minimization problem \eqref{pbB} has an energy strictly less than the minimizer $\u$, which is a contradiction.

\emph{Case 2)} Assume that $\operatorname{int}(\overline{\omega^*\cup \omega^{**}} )$ is not connected. First, since $\{\varphi_1>0\}$ is connected, it follows that either $\{\varphi_1>0\} \subset \overline{\omega^*}$ or $\{\varphi_1>0\} \subset \overline{\omega^{**}}$. Let us suppose that the first condition holds. Since $\omega^{**}$ is an $N$-dimensional open set, its boundary has dimension at least $N-1$, which means that it contains some regular points. Therefore there exist another index $j \in \{2,\dots,k\}$, say $j=2$, and a connected component of $\omega_2$, which we denote by $\omega_2'$, such that $\omega^{**}$ is adjacent to $\omega_2'$; in particular, $A_2 := \operatorname{int}(\overline{\omega^{**}\cup \omega_2'} )$ is connected (it could be that $\omega_2'=\omega_2$, but we cannot exclude that also $\omega_2$ and any other $\omega_i$ are disconnected). At this point we define a function $\bar u_2 \in H^1(\Omega)$ so that $\bar u_2|_{A_2}$ is a minimizer of \eqref{eq:conn:I} in $A_2$, under the boundary condition $\bar u_2 = \varphi_2$ on $\partial \Omega \cap \overline{A_2}$ and $\bar u_2=0$ on $\partial  A_2\setminus (\partial \Omega \cap \overline{A_2})$; and, outside of $A_2$, we set $\bar u_2=u_2$ in $\Omega \setminus A_2$. Then, arguing as in Case 1, by the strong maximum principle (and exploiting the fact that $\lambda_1=\lambda_2=\bar \lambda$) it follows that
\[
\int_{A_2} (|\nabla \bar u_2|^2 - 2\bar\lambda \bar u_2) \, dx < \int_{\omega^{**}} (|\nabla u_1|^2 - 2\bar\lambda u_1) \, dx 
+
\int_{\omega_2'} (|\nabla u_2|^2 - 2\bar\lambda u_2) \, dx,
\]
and hence the competitor $(u_1|_{\omega^*}, \bar u_2, u_3, \dots, u_k)$ has strictly lower energy than the original minimizer $\u$, which is a contradiction. Thus, the proof is complete.
\end{proof}

\section*{Acknowledgment} 
The authors are research fellows of Istituto Nazionale di Alta Matematica INDAM group GNAMPA. G.F. and G.T. are supported by the GNAMPA project "Struttura fine e regolarit\`a in problemi variazionali non-lineari" codice CUP E53C25002010001. 

\medskip

\noindent \textbf{Data availability:} No data were used for the research described in the article.

\medskip

\noindent \textbf{Conflict of interest:} The authors declare that they have no conflict of interest.

\end{document}